\documentclass[leqno,12pt,oneside]{amsart}
\usepackage{graphicx}
\usepackage{mathtools}

\usepackage{amsmath,amsthm,amsfonts,amssymb,latexsym,amscd,enumerate}
\usepackage[a4paper, total={6in, 9in}]{geometry}
\usepackage[all]{xy}
\usepackage{xy}
\usepackage{xcolor}
\usepackage{boxedminipage}

\theoremstyle{plain}
\newtheorem{theorem}{Theorem}[section]
\newtheorem{lemma}[theorem]{Lemma}
\newtheorem{corollary}[theorem]{Corollary}
\newtheorem{proposition}[theorem]{Proposition}

\theoremstyle{definition}
\newtheorem{definition}[theorem]{Definition}

\newtheorem{remark}[theorem]{Remark}

\newtheorem*{thm:obstruction}{Theorem \ref{thm:obstruction}}
\newtheorem*{thm:CalliasKKClass}{Theorem \ref{thm:CalliasKKClass}}
\newtheorem*{thm:infiniteKofCorona}{Theorem \ref{thm:infiniteKofCorona}}

\theoremstyle{remark}

\newcommand{\norm}[1]{\left\lVert#1\right\rVert}

\DeclareMathOperator{\ch}{ch}

\newcommand{\dirac}{\partial \!\!\! \slash}

\DeclareMathOperator{\End}{End}

\DeclareMathOperator{\ind}{index}

\begin{document}   
	\sloppy
	\begin{abstract}
	We formulate, for any Lie group $G$ acting isometrically on a manifold $M$, the general notion of a $G$-equivariant elliptic operator that is invertible outside of a $G$-cocompact subset of $M$. We prove a version of the Rellich lemma for this setting and use this to define the equivariant index of such operators. We show that $G$-equivariant Callias-type operators are self-adjoint, regular, and hence equivariantly invertible at infinity. Such operators explicitly arise from a pairing of the Dirac operator with the equivariant Higson corona. We apply the theory developed herein to obtain an obstruction to positive scalar curvature metrics on non-cocompact manifolds.
	\end{abstract}
	
	\title[Equivariant Callias Index and PSC]
	{Index of Equivariant Callias-Type Operators and Invariant Metrics of Positive Scalar Curvature}
	
	\author{Hao Guo}
	
	
	\maketitle
	\section{Introduction}
	It is well-known that a Dirac operator $D$ on a non-compact manifold $M$ is not in general Fredholm, since the usual version of the Rellich lemma fails in this setting. Nevertheless, it is possible to modify $D$ so as to make it Fredholm but still remain within the class of Dirac-type operators. One such modification is a \textit{Callias-type operator}, which was initially studied by Callias \cite{Callias} on $M$ a Euclidean space, before being generalised to the setting of Riemannian manifolds by others \cite{BottSeeley},\cite{Anghel},\cite{BruningMoscovici},\cite{Bunke}. A Callias-type operator may be written as $B=D+\Phi$, where $\Phi$ is an endomorphism making $B$ invertible at infinity. As in \cite{Bunke}, one may form the order-$0$ bounded transform $F$ of $B$, defined formally by 
	$$F\coloneqq B(B^2+f)^{-1/2},$$
	where $f$ is a compactly supported function. The formal computation $F^2-1=-f(B^2+f)^{-1}$ shows that $F$ is a Fredholm operator, since multiplication by the compactly supported function $f$ defines a compact operator between Sobolev spaces $H^i\rightarrow H^j$ for all $i>j$.
	
	More generally, if $M$ is a manifold with a proper, non-cocompact action of a Lie group $G$, a $G$-invariant Dirac operator on $M$ need not be Fredholm in the sense of $C^*$-algebras. From the point of view of the equivariant index map in $KK$-theory, the lack of a general index for elements of the equivariant analytic $K$-homology $K_*^G(M)$ can be traced back to the lack of a counterpart to the canonical projection in $KK(\mathbb{C},C_0(M)\rtimes G)$ defined by a compactly supported cut-off function when $M/G$ is compact.
	
	In this paper we introduce $G$-equivariant analogues of the Sobolev spaces $H^i$ and the Rellich lemma. With these tools, the formal computation above can be made to work in the non-cocompact setting for any operator whose square is positive outside of a cocompact set. We establish that such operators have a $G$-equivariant index. 
	
	This provides an equivariant generalization of the situation considered by Roe in \cite{Roe}. It is possible to formulate the index theory of equivariant Callias-type operators in terms of the Roe algebra. We will do this, as well as consider applications, in forthcoming papers.
	
	\vspace{0.5cm}
	
	In section \ref{sec:G-Sobolev}, we define {\it $G$-Sobolev modules} $\mathcal{E}^i$ over the (maximal or reduced) group $C^*$-algebra $C^*(G)$. The main purpose of this construction is to establish the analogue of the Rellich lemma alluded to above.
	
	In section \ref{sec:G-Invertible} we define the notion of {\it $G$-invertibility at infinity} - an equivariant generalisation of the notion of invertibility at infinity (see for example \cite{Bunke} section $1$).
	
	\begin{thm:CalliasKKClass}
		Suppose $B$ is $G$-invertible at infinity. Let $F$ be its bounded transform. Then $(\mathcal{E}^0,F)$ is a Kasparov module over the pair of $C^*$-algebras $(\mathbb{C},C^*(G))$. The class $[\mathcal{E}^0,F]\in KK(\mathbb{C},C^*(G))$ is independent of the choice of the cocompactly supported function $f$ used to define $F$.
	\end{thm:CalliasKKClass}

	This implies that $G$-invertible-at-infinity operators are equivariantly Fredholm. An example of such an operator is a {\it $G$-Callias-type operator}, defined in section \ref{sec:G-Callias}. 
	
	We show that these operators are essentially self-adjoint and regular in the sense of Hilbert modules. We give an explicit construction of $G$-Callias-type operators in section \ref{sec:Phi} using the $K$-theory of an equivariant Higson corona of $M$, whose $K$-theory turns out to be highly non-trivial. In particular, we prove:
	\begin{thm:infiniteKofCorona}
		Let $M$ be a complete $G$-Riemannian manifold with  $M/G$ non-compact. Then the $K$-theory of the Higson $G$-corona of $M$ is uncountable.
	\end{thm:infiniteKofCorona}
	
	In section \ref{sec:PSC} we give an application of the theory to $G$-invariant metrics of positive scalar curvature on non-cocompact $M$. We prove:
	\begin{thm:obstruction}
		Let $M$ be a $G$-equivariantly spin Riemannian manifold with $G$-spin-Dirac operator $\dirac_0$. If $M$ admits a $G$-invariant metric with pointwise positive scalar curvature, then the $G$-index of every $G$-Callias-type operator on $M$ vanishes. Let $F$ be the bounded transform of the $G$-Callias-type operator defined by $\dirac_0$ and any $G$-admissible endomorphism $\Phi$. Then
		$$\ind_G(F)=0\in K_0(C^*(G)).$$
	\end{thm:obstruction}
	This result vastly generalises two existing results on obstructions to $G$-invariant positive scalar curvature on proper $G$-spin manifolds, where $G$ is a non-compact Lie group. The first is a recent result of Zhang (Theorem 2.2 of \cite{WeipingShort}), which was the first generalisation of Lichnerowicz' vanishing theorem to the cocompact-action case; the notion of index used there was the Mathai-Zhang index (see \cite{MZ}). The second is Theorem 54 of \cite{GMW}, which states that the equivariant index of a $G$-invariant Dirac operator vanishes in the presence of $G$-invariant positive scalar curvature.
	
	The methods and results of this paper can be contrasted with the equivariant index theory and applications studied in \cite{GMW}. First, \cite{GMW} deals exclusively with index theory in the cocompact case, where the $C^*(G)$-Fredholmness of the Dirac operator was known. In addition, \cite{GMW} focused almost entirely on the case of almost-connected $G$, where a global slice of the manifold exists. 
	
	The paper \cite{Cecchini} also deals with Callias-type operators in the setting of Hilbert modules on non-compact manifolds. The indices studied there and in the present paper are in general different generalisations of the classical Callias-type index, although the case when $G$ is the fundamental group of a compact manifold can be approached from both directions. The technical difference between the two approaches is that whereas in \cite{Cecchini}, the Hilbert module structure arises by twisting the space of $L^2$-sections of a vector bundle by a Hilbert module bundle, the $G$-Sobolev modules we define here arise from the $G$-action on the space of compactly supported smooth sections, and then completing with respect to a particular inner product with values in the group $C^*$-algebra. Notably, \cite{Cecchini} establishes a twisted Callias-type index theorem that allows one to reduce the computation of a Callias-type index to an index on a compact hypersurface, using the machinery of $KK$-theory. The analogous theorem in the $G$-equivariant case, although expected to be true, does not follow directly from the results of \cite{Cecchini}. We hope to address this in future work.
	
	Finally, other versions of $G$-index theory, for non-compact $M/G$ and $G$, have been developed elsewhere. This includes the work of Hochs-Mathai \cite{HochsMathai},\cite{HochsMathai2}, Braverman \cite{Braverman}, and Hochs-Song \cite{Hochs-Song1},\cite{Hochs-Song2},\cite{Hochs-Song3}.\\

The results contained in this paper set the stage for a number of further questions. The first concerns the Baum-Connes conjecture. When $M/G$ is compact, the equivariant index map for $G$-Callias-type operators reduces to the Baum-Connes assembly map. However, when $M/G$ is non-compact, the class of operators encapsulated by Callias index theory is strictly larger, and an interesting open question arises:\\

\noindent\textbf{Open question:} Does the index of a $G$-Callias-type operator always lie in the image of the Baum-Connes assembly map?\\

Another direction in which we will look to apply the theory developed herein is the problem of \emph{quantisation commutes with reduction} (or $[Q,R]=0$) for non-cocompact manifolds. In \cite{Landsman}, Landsman used the language of noncommutative geometry to formulate a version of geometric quantization for $G$-cocompact Spin$^c$-manifolds. He defined quantization as taking equivariant index of a Spin$^c$-Dirac operator, via the equivariant index map $\ind_G:K_0^G(M)\rightarrow K_0(C^*(G)),$ while reduction of a quantized system is defined to be the image of the index under a certain tracial map 
$$R_G:K_0(C^*(G))\rightarrow K_0(\mathbb{C})\cong\mathbb{Z}.$$ 
Landsman conjectured that, under certain assumptions on the symmetry group $G$ (such as semi-simplicity) $[Q,R]=0$ holds: 
$$R_G(\ind_G(D_M))=\textnormal{index}(D_{M_0}).$$
Here $M_0$ is a certain (compact) reduced space of $M$ defined as the inverse image of the regular value $0$ under the symplectic moment map $\mu$, and the index on the right-hand side of the equation is the classical Fredholm index. 

This conjecture was solved by Mathai and Zhang \cite{MZ} in 2008, following work done in special cases by Hochs \cite{HochsQuantisation} and Hochs-Landsman \cite{HochsLandsman}.\\

The theory we formulate in this paper provides a natural class of operators with which to investigate the $[Q,R]=0$ problem in the non-cocompact setting, namely $G$-equivariant Callias-type operators that arise from an equivariant Spin$^c$ structure on $M$. In future work, we aim to address the following specific question.\\

\noindent\textbf{Open question:} Does $[Q,R]=0$ hold for $G$-Spin$^c$ Callias-type operators?\\

\subsection*{Acknowledgements}
H.G. would like to thank his supervisors, Mathai Varghese and Hang Wang, for their guidance in this project, and would like to acknowledge the support provided by a University of Adelaide Divisional Scholarship. Additionally, H.G. would like to thank the referees, whose comments have clarified the presentation of this paper.
	\vspace{1cm}

	\tableofcontents
	\vspace{1cm}

	\section{Notation and Terminology}
	\label{sec:Notation}
	
	Throughout this paper, $M$ will be a Riemannian manifold on which a Lie group $G$ acts smoothly, properly and isometrically. We shall call this set-up {\it $G$-Riemannian manifold}. If the orbit space $M/G$ is compact, we say that the $G$-action is {\it cocompact}, or that $M$ is a {\it cocompact $G$-manifold}. Except in section \ref{sec:PSC}, we will suppress the Riemannian metric when mentioning $M$, to avoid unnecessary confusion with elements of $G$.
	
	Fix a left Haar measure $dg$ on $G$, with modular function $\mu\colon G\rightarrow\mathbb{R}^+$ given by $d(gs) = \mu(s)dg$. We use $d\mu$ to denote the smooth $G$-invariant measure on $M$ induced by the Riemannian metric.
	
	Let $\pi\colon M\rightarrow M/G$ be the natural projection. A subset $K\subseteq M$ is called {\it cocompact} if $\pi(K)$ has compact closure in $M/G$. We say that $S\subseteq M$ is {\it cocompactly compact} if for any cocompact $K\subseteq M$, $S\cap K$ has compact closure in $M$. A proper $G$-manifold $M$ admits a smooth {\it cut-off function} $\mathfrak{c}\colon M\rightarrow [0,1]$, with the property that for all $x\in M$,
	$$\int_G\mathfrak{c}(g^{-1}x)\,dg = 1.$$
	Clearly, supp$(\mathfrak{c})$ is cocompactly compact.
	
	Let $\pi_1,\pi_2\colon M\times M\rightarrow M$ be the projection maps onto the first and second factors. Given an operator $A$ on $M$ with Schwartz kernel $k_A$, we say that $k_A$ has {\it compact} (resp. {\it cocompactly compact}) support if there exists a compact (resp. cocompactly compact) subset $\tilde{K}\subseteq M$ such that supp$(k_A)\subseteq\tilde{K}\times\tilde{K}$.
	
	For $A$ a $C^*$-algebra, denote its positive (which we use to mean positive semi-definite) elements by $A_+$. For Hilbert $A$-modules $\mathcal{M}$ and $\mathcal{N}$, denote the bounded adjointable operators and compact operators $\mathcal{M}\rightarrow\mathcal{N}$ by $\mathcal{L}(\mathcal{M},\mathcal{N})$ and $\mathcal{K}(\mathcal{M},\mathcal{N})$ respectively. If $\mathcal{M}=\mathcal{N}$, we use $\mathcal{L}(\mathcal{M})$ and $\mathcal{K}(\mathcal{M})$ respectively.
		\vspace{1cm}
	
	\section{Sobolev Modules}
	\label{sec:G-Sobolev}
	We first set up certain generalisations of Sobolev spaces that take into account the $G$-action. The definition is based on Kasparov's definition of the module $\mathcal{E}$ (see for instance \cite{Kasparov} section 5), which is a $C^*(G)$-module analogue of $L^2(E)$.
	
	First let us recall the definition of $\mathcal{E}$. Let $E$ be a $G$-equivariant Hermitian vector bundle over a non-cocompact $G$-Riemannian manifold $M$. The space of compactly supported smooth sections $C_c^\infty(E)$ can be given a pre-Hilbert $C_c^\infty(G)$-module structure, with right $C_c^\infty(G)$-action and $C_c^\infty(G)$-valued inner product given by
	$$(e\cdot b)(x) = \int_G g(e)(x)\cdot b(g^{-1})\cdot\mu(g)^{-1/2}\,dg\in C_c^\infty(E),$$
	$$\langle e_1,e_2\rangle(g)=\mu(g)^{-1/2}\int_M\langle e_1(x),g(e_2)(x)\rangle_E\,d\mu\in C_c^\infty(G),$$
	for $e,e_1,e_2\in C_c^\infty(E)$ and $b\in C_c^\infty(G)$. Here $G$ acts on $C_c(E)$ by $g(e)(x)\coloneqq g(e(g^{-1}x)).$ Then $\mathcal{E}$ is defined to be the completion of $C_c^\infty(E)$ under the norm induced by the above inner product. The $C_c^\infty(G)$-action extends to a $C^*(G)$-action on $\mathcal{E}$ and the inner product to a $C^*(G)$-valued inner product. Thus $\mathcal{E}$ is a Hilbert $C^*(G)$-module.
	
	We generalise this definition using Sobolev-type inner products defined using a Dirac-type operator, by which we mean an operator whose principal symbol is that of a Dirac operator. In particular, this includes the Dirac-plus-potential-type operators of section \ref{sec:G-Callias}.
	
	\begin{definition}
		\label{def:Sobolev}
		Let $M$ be a (not necessarily cocompact) $G$-Riemannian manifold and $E$ a Hermitian $G$-vector bundle over $M$. Let $B$ be a $G$-invariant, formally self-adjoint Dirac-type operator on $E$ with initial domain $C_c^\infty(E)$. For each integer $i\geq 0$, let $C_c^{\infty,i}(E)$ be the pre-Hilbert $C_c^\infty(G)$-module with right $C_c^\infty(G)$-action and $C_c^\infty(G)$-valued inner product given by 
		$$(e\cdot b)(x) = \int_G g(e)(x)\cdot b(g^{-1})\mu(g)^{-1/2}\,dg\in C_c^\infty(E),$$
		$$\langle e_1,e_2\rangle_i(g)=\mu(g)^{-1/2}\sum_{k=0}^i\int_M\langle B^k e_1(x),B^k g(e_2)(x)\rangle_E\,d\mu\in C_c^\infty(G),$$
		for $e,e_1,e_2\in C_c^\infty(E)$ and $b\in C_c^\infty(G)$, and where we set $B^0$ equal to the identity operator. Positivity of these inner products is proved in Lemma \ref{lem:Positivityofi}. Denote by $\mathcal{E}^i(E)$, or simply $\mathcal{E}^i$, the vector space completion of $C_c^{\infty,i}(E)$ with respect to the norm induced by $\langle\,\,,\,\,\rangle_i$, and extend naturally the $C_c^\infty(G)$-action to a $C^*(G)$-action and $\langle\,\,,\,\,\rangle_i$ to a $C^*(G)$-valued inner product, to give $\mathcal{E}^i$ the structure of a Hilbert $C^*(G)$-module. Let $\norm{\,\cdot\,}_i$ denote the associated norm. We call $\mathcal{E}^i$ the {\it $i$-th $G$-Sobolev module with respect to $B$}. 
	\end{definition}
	\begin{lemma}
		\label{lem:Positivityofi}
		For each $i\geq 0$, the $C_c^\infty(G)$-valued inner product $\langle\,\,,\,\,\rangle_i$ on the pre-Hilbert $C_c^\infty(G)$-module $C_c^{\infty,i}(E)$ is positive in $C^*(G)$.
	\end{lemma}
	\begin{proof}
		For any $u\in C_c^{\infty,i}(E)$ we can write
		$\langle u,u\rangle_i=\sum_{k=0}^i\langle B^k u,B^k u\rangle_0\in C_c^\infty(G)\subseteq C^*(G).$ Each summand is in $C^*(G)_+$, as shown in \cite{Kasparov} section 5, and a sum of finitely many positive elements in a $C^*$-algebra is positive.
	\end{proof}
	\begin{definition}
		\label{def:triple}
		A {\it  $G$-triple} $(G,M,E)$ consists of a Lie group $G$, a proper $G$-Riemannian manifold $M$ and a Hermitian $G$-vector bundle $E\rightarrow M$, together with a $G$-invariant Dirac-type operator $B$ on $E$ and the collection $\{\mathcal{E}^i\}$ of $G$-Sobolev modules formed using $B$. If $E=E^+\oplus E^-$ is $\mathbb{Z}_2$-graded and $B=B^-\oplus B^+$ is an odd operator, we shall call $(G,M,E)$ a $\mathbb{Z}_2$-graded $G$-triple. If $M$ is cocompact, we shall say the triple $(G,M,E)$ is cocompact.
	\end{definition}
	\begin{remark}
	In general, the inner products and modules $\mathcal{E}^i$ in Definitions \ref{def:Sobolev} and \ref{def:triple} depend on the choice of the operator $B$.
	\end{remark}
	Given a $G$-triple with operator $B$, $\textnormal{dom}(\overline{B^i})$ equipped with the graph norm is isomorphic to $\mathcal{E}^i$. Thus $\overline{B^i}$ is a bounded operator $\mathcal{E}^i\rightarrow\mathcal{E}^0$ for all $i\geq 0$; Proposition \ref{prop:Bbounded} implies that $\overline{B^i}$ is adjointable. Further, the results of subsection \ref{subsec:essentialregularity} imply that $B$ is regular and essentially self-adjoint. Therefore, except in those subsections, we will not make the distinction between $B$ and its closure $\overline{B}$. By $B^i$ we will mean the bounded adjointable operator $\overline{B^i}:\mathcal{E}^i\rightarrow\mathcal{E}^0$.
	
	\subsection{Boundedness and Adjointability}
	
	We now establish basic boundedness and adjointability results for $G$-Sobolev modules. When $M/G$ is compact, Kasparov (\cite{Kasparov} Theorem 5.4) proved that an $L^2(E)$-bounded, $G$-invariant operator on $C_c(E)$ with properly supported Schwartz kernel defines a bounded adjointable operator on $\mathcal{E}=\mathcal{E}^0$. We now use the same method of proof to establish the following result.
	\begin{proposition}
		\label{prop:boundedhilbert}
		Let $(G,M,E)$ be a cocompact $G$-triple with operator $B$. Let $A$ be an operator on $C_c^\infty(E)$ that is $G$-invariant and bounded $H^i(E)\rightarrow H^j(E)$, where $H^k(E)$ denotes the completion of $C_c^\infty(E)$ with respect to the inner product $\langle B^k e_1,B^k e_2\rangle_{L^2(E)}$, for $e_1,e_2\in C_c^\infty(E)$.  If $A$ has properly supported Schwartz kernel, then $A$ defines an element of $\mathcal{L}(\mathcal{E}^i,\mathcal{E}^j)$ with norm $\leq C\cdot\norm{A}_{B(H^i,H^j)}$, where $C$ is a constant that depends only on the supports of $\mathfrak{c}A^*A+A^*A\mathfrak{c}$ and $\mathfrak{c}AA^*+AA^*\mathfrak{c}$.
	\end{proposition}
	\begin{remark}
		\label{rem:unitary}
		The action of $G$ on $C_c^\infty(E)$ extends to a unitary action on each $H^i(E)$. For an operator $T$ and $g\in G$, we define
		$$g(T):=gTg^{-1}.$$
		In particular, we may speak of $G$-invariant operators $H^i(E)\rightarrow H^j(E)$.
	\end{remark}
	The proof of Proposition \ref{prop:boundedhilbert} uses the following lemma.
	\begin{lemma}
		Let $(G,M,E)$ be a cocompact $G$-triple. Let $T$ be a bounded positive operator on $H^i(E)$ with compactly supported Schwartz kernel. Then for $e\in C_c^\infty(E)$, $$\left\langle e,\left(\int_G s(T)\,ds\right)(e)\right\rangle_{\mathcal{E}^i}\in C^*(G)_+.$$ 
	\end{lemma}
	\begin{proof}
		The proof we give is a sketch based on \cite{Kasparov} Lemma 5.3. Note that $T$ has a unique positive square root $T^{1/2}\colon H^i(E)\rightarrow H^i(E)$, so that
		$$\langle s(e),T(s(e)\rangle_{H^i(E)}=\langle T^{1/2}(s(e)),T^{1/2}(s(e))\rangle_{H^i(E)}$$
		for $e\in C_c(E)$ and $s\in G$. The function $G\rightarrow H^i(E)$ given by $s\mapsto T^{1/2}(s(e))$ has compact support in $G$. It can be shown as in \cite{Kasparov} Lemma 5.3 that, for any unitary representation of $G$ on a Hilbert space $H$ and any $h\in H$, 
		$$v\coloneqq\int_G \mu^{-1/2}(s)T^{1/2}(s(e))\otimes s(h)\,ds$$ 
		is a well-defined vector in $H^i(E)\otimes H$, and that $v$ has norm equal to
		\begin{align*}
		&\int_G\int_G\mu^{-1/2}(t)\mu^{-1/2}(s)\langle T^{1/2}s(e),T^{1/2}(t(e))\rangle_{H^i(E)}\cdot\langle s(h),t(h)\rangle_H\,ds\,dt\\
		&=\int_G\bigg\langle e,\left(\int_Gs(T)\,ds\right)(e)\bigg\rangle_{\mathcal{E}^i}(t)\cdot\langle h,t(h)\rangle_H\,dt.
		\end{align*}
		Thus $\langle e,(\int_G s(T)\,ds)(e)\rangle_{\mathcal{E}^i}$ is a positive operator on $H$ for all unitary representations of $G$, where we let $f\in C_c^\infty(G)$ act on $H$ by $f(h)\coloneqq\int_G f(g)g(h)\,dg.$ It follows that $\langle e,(\int_G s(T)\,ds)(e)\rangle_{\mathcal{E}^i}$ is a positive element of $C^*(G)$.
	\end{proof}
	\begin{proof}[Proof of Proposition~\ref{prop:boundedhilbert}] 
		Let $\mathfrak{c}$ be a cut-off function on $M$. The operator $A_1\coloneqq(\mathfrak{c}A^*A+A^*A\mathfrak{c})/2$ is bounded on $H^i(E)$ by $\norm{A}^2\norm{\mathfrak{c}}$, where $\norm{A}$ is the norm of $A:H^i(E)\rightarrow H^j(E)$. Note that $A_1$ is self-adjoint with compactly supported Schwartz kernel. Let $\mathfrak{c}_1$ be a non-negative, compactly supported function on $M$, identically $1$ on the support of the kernel of $A_1$. Then the operator $A_2\coloneqq \mathfrak{c}_1^2\norm{A}^2\norm{\mathfrak{c}}-A_1$ is positive, bounded and has Schwartz kernel with compact support. Now let $M_0$ be the least upper bound over all functions $\mathfrak{c}_1$ satisfying the above conditions of the function $x\mapsto\int_G g(\mathfrak{c}_1^2)(x)\,dg$ on $M$. Note that $M_0$ depends only on the support of $\mathfrak{c}A^*A+A^*A\mathfrak{c}$. Since $\int_G g(A_1)\,dg = A^*A$, the previous lemma applied to $A_2$ shows that for any $e\in C_c^\infty(E)$,
		$$\bigg\langle e,\left(\int_G g(A_2)\,dg\right)e\bigg\rangle_{\mathcal{E}^i}=\left(\int_Gg(\mathfrak{c}_1^2)\norm{A}^2\norm{\mathfrak{c}}\,dg\right)\langle e,e\rangle_{\mathcal{E}^i} - \langle e,A^*A(e)\rangle_{\mathcal{E}^i}$$
		is positive in $C^*(G)$. Hence
		$\langle A(e),A(e)\rangle_{\mathcal{E}^j}=\langle e,A^*A(e)\rangle_{\mathcal{E}^i}\leq M_0\norm{A}^2\norm{\mathfrak{c}}\langle e,e\rangle_{\mathcal{E}^{i}}\in C^*(G)$, and $A$ extends to an operator on all of $\mathcal{E}^i$. Similarly, $A^*\colon H^j(E)\rightarrow H^i(E)$ defines a bounded operator $\mathcal{E}^j(E)\rightarrow\mathcal{E}^i(E)$ that one checks is the adjoint of $A$.
	\end{proof}
	\begin{proposition}
		\label{prop:Bbounded}
		Let $(G,M,E)$ be a (not necessarily cocompact) $G$-triple with $B$ as above. For $j\geq 0$, $B$ defines an element of $\mathcal{L}(\mathcal{E}^{j+1},\mathcal{E}^j)$.
	\end{proposition}
	\begin{proof}
		Boundedness is clear. Now since $G$ acts on $M$ properly, there exists a countable, locally finite open covering $\mathcal{U}$ of $M$ by $G$-stable open subsets $U_k$, $k\in\mathbb{N}$, such that for each $k$, $U_k$ is cocompact and $\overline{U_k}$ is a manifold with boundary (see \cite{Neeb} Lemma IV.4. for such an exhaustive sequence in the non-equivariant case). By \cite{Palais} (see also \cite{PalaisBook} Theorem 5.2.5) one can find a $G$-invariant partition of unity $\{\rho_k\}$ subordinate to $\mathcal{U}$. Now one can form the modules $\mathcal{E}^i$ by first forming analogous local modules $\mathcal{E}^i_{U_k}$ on $U_k$, where $U_k$ is considered as an open $G$-submanifold of a cocompact $G$-manifold without boundary, namely the double $\overline{U_k}^+$ of the cocompact $G$-manifold with boundary $\overline{U_k}$. (This can be done since there exists a $G$-equivariant collar neighbourhood of $\partial\overline{U_k}$ inside $\overline{U_k}$, by Theorem 3.5 of \cite{Kankaanrinta}.) One can then use $\{\rho_k\}$ to form the inner product on $\mathcal{E}^i$. For example, in the case of $\mathcal{E}^0$, we have
		\begin{align*}
		\langle s,t\rangle_{\mathcal{E}^0}(g)&=\mu(g)^{-1/2}\sum_{k}\langle\sqrt{\rho_k}s,\sqrt{\rho_k}gt\rangle_{L^2\left(E|_{U_k}\right)}\\&=\sum_{k}\langle\sqrt{\rho_k}s,\sqrt{\rho_k}t\rangle_{\mathcal{E}^0_{U_k}}(g),
		\end{align*}
		for $s,t\in C_c^\infty(E)$ and $g\in G$, where we have used $G$-invariance of $\rho_k$. By Proposition \ref{prop:boundedhilbert}, the operator $B$ restricted to sections supported on each neighbourhood $U_k$ is in $\mathcal{L}(\mathcal{E}^{j+1}_{U_k},\mathcal{E}^j_{U_k}).$ By \cite{Kasparov} Theorem 5.8, the local inverse $((B^2+1)_{U_k})^{-1}\colon\mathcal{E}^l_{U_k}\rightarrow\mathcal{E}^{l+2}_{U_k}$ exists for all $l\geq 0$. One can verify that $B_{U_k}((B^2+1)_{U_k})^{-1}$ is the adjoint of $B_{U_k}$, and that $\sum_{k}B_{U_k}((B^2+1)_{U_k})^{-1}\rho_k$ is the adjoint of $B$.
	\end{proof}
	\begin{corollary}
		\label{prop:Bibounded}
		Let $(G,M,E)$ be a $G$-triple with $B$ as above. For $j\geq 0$, $B^i$ defines an element of $\mathcal{L}(\mathcal{E}^{j+i},\mathcal{E}^j)$.
	\end{corollary}
	\begin{proposition}
		\label{prop:multiplicationbounded}
		Let $(G,M,E)$ be a $G$-triple. Then multiplication by a $G$-invariant function $f\colon M\rightarrow\mathbb{C}$ for which $\norm{f}_\infty<\infty$ is an element of $\mathcal{L}(\mathcal{E}^i,\mathcal{E}^0)$ for all $i\geq 0$.
	\end{proposition}
	\begin{proof}
		Boundedness follows from 
		$$\norm{\langle fe,fe\rangle_{\mathcal{E}^0}}_{C^*(G)} \leq C^2\norm{\langle e,e\rangle_{\mathcal{E}^0}}_{C^*(G)}\leq C^2\norm{\langle e,e\rangle_{\mathcal{E}^i}}_{C^*(G)}.$$ Now let $f^*\colon H^0(E)\rightarrow H^i(E)$ be the adjoint of $f\colon H^i(E)\rightarrow H^0(E)$. Since $f\colon H^i(E)\rightarrow H^0(E)$ is bounded and $G$-invariant, one sees that $f^*$ is $G$-invariant. For $e_1,e_2\in C_c^\infty(E)$ and $g\in G$, we have
		\begin{align*}
		\langle fe_1,e_2\rangle_{\mathcal{E}^0} (g)&=\mu(g)^{-1/2}\langle fe_1,g(e_2)\rangle_{H^0(E)}\\
		&=\langle e_1,f^*e_2\rangle_{\mathcal{E}^i}(g).
		\end{align*}
		Hence $f^*$ is the adjoint for $f\colon\mathcal{E}^i\rightarrow\mathcal{E}^0$ and therefore bounded \cite{Lance}.
	\end{proof}

	\subsection{An Equivariant Rellich Lemma} Recall the following non-compact analogue of the Rellich lemma:
	\begin{lemma}
		\label{lem:noncompactrellich} Let $M$ be a non-compact manifold and $f\colon M\rightarrow\mathbb{C}$ a compactly supported function. Then multiplication by $f$ is a compact operator $H^s(M)\rightarrow H^t(M)$ if $s>t$.
	\end{lemma}
	
	Now suppose $(G,M,E)$ is a non-cocompact $G$-triple with operator $B$, and let $H^i(E)$ denote the completion of $C_c^\infty(E)$ with respect to the inner product $\langle B^i e_1,B^i e_2\rangle_{L^2(E)}$, for $e_1,e_2\in C_c^\infty(E)$.
	
	One can verify, from the definition of rank-one operators between Hilbert modules (see \cite{Manuilov}) that a rank-one element of $\mathcal{K}(\mathcal{E}^s,\mathcal{E}^t)$ can be constructed by taking the \emph{$G$-average} of a rank-one operator $H^s(E)\rightarrow H^t(E)$ in the sense of Hilbert spaces. More precisely, if $e_1$ and $e_2$ are compactly supported smooth sections of $E$, the $G$-average of the rank-one operator $\theta_{e_1,e_2}:H^s(E)\rightarrow H^t(E)$ is defined to be the operator
	$$\int_G g(\theta_{e_1,e_2})\,dg:\mathcal{E}^s\rightarrow\mathcal{E}^t$$
	that takes an element $e\in C_c^\infty(E)$ to the element of $C_c^\infty(E)$ given by
	\begin{align*}
	\left(\int_G g(\theta_{e_1,e_2})\,dg\right)(e)(x)&:=\int_M\left(\int_G\theta_{g(e_1)(x),g(e_2)(y)}\,dg\right)e(y)\,d\mu\\
	&=\int_M\int_G g(e_1)(x)\langle gB^s e_2(y),B^s e(y)\rangle_{E_y}\,dg\,d\mu.
	\end{align*}
	More generally, if $\theta:H^s(E)\rightarrow H^t(E)$ is the sum of finitely many such rank-one operators, we can make sense of the $G$-average of $\theta$, denoted by
	$$\int_G g(\theta)\,dg.$$
	We now prove the following equivariant version of the Rellich lemma, which will be important our subsequent analysis.
	\begin{theorem}
		\label{thm:G-Rellich}
		Let $f$ be a cocompactly supported $G$-invariant function. Then multiplication by $f$ is an element of $\mathcal{K}(\mathcal{E}^{s},\mathcal{E}^{t})$ for $s>t$.
	\end{theorem}
	\begin{proof}
		First observe that, by Proposition \ref{prop:multiplicationbounded}, multiplication by $f$ is an element of $\mathcal{L}(\mathcal{E}^i,\mathcal{E}^0)$. To see that it is compact, let $\mathfrak{c}$ be a cut-off function. By Lemma \ref{lem:noncompactrellich}, multiplication by $\mathfrak{c}f$ is a compact operator $H^s(E)\rightarrow H^t(E)$ and hence the operator-norm limit of a sequence of finite-rank operators $(\theta_i)_{i\in\mathbb{N}}$. We may assume that the $\theta_i$ have continuous Schwartz kernels supported within a fixed compact subset $L\subseteq M\times M$ of diameter $r$. The averaged operator $\int_G g(\theta_i)\,dg$ is still bounded $H^s(E)\rightarrow H^t(E)$, has Schwartz kernel supported within an $r$-ball of the diagonal, and we have 
		$$\int_G g(\theta_i)\,dg\rightarrow\int_G g(\mathfrak{c}f)\,dg=f$$ 
		in the operator norm on $B(H^s,H^t)$. Since both $f$ and $\int_G g(\theta_i)\,dg$ have properly and cocompactly supported Schwartz kernels, Proposition \ref{prop:boundedhilbert} applies and shows that the convergence also holds in the norm of $\mathcal{L}(\mathcal{E}^i,\mathcal{E}^j)$.
	\end{proof}
	\vspace{1cm}
	\section{$G$-invertibility at Infinity}
	\label{sec:G-Invertible}
	In order to formulate index theory, we introduce a notion of invertibility at infinity \cite{Bunke} for the non-cocompact $G$-setting. Let $(G,M,E)$ be a $\mathbb{Z}_2$-graded non-cocompact $G$-triple with operator $B$. It follows from our previous results that $B^2\in\mathcal{L}(\mathcal{E}^2,\mathcal{E}^0).$ Moreover, 
	Proposition \ref{prop:multiplicationbounded} gives:
	\begin{lemma}
		Let $f\colon M\rightarrow\mathbb{C}$ be a continuous $G$-invariant function for which $\norm{f}_\infty<\infty$. Then $B^2+f\in\mathcal{L}(\mathcal{E}^2,\mathcal{E}^0)$.
	\end{lemma}
	Next recall that an unbounded operator $A$ on a Hilbert module $\mathcal{M}$ is said to be {\it regular} if its graph is orthogonally complementable in $\mathcal{M}\oplus\mathcal{M}$. It can be shown that $A$ is both regular and self-adjoint if and only if there exists $\mu\in i\mathbb{R}$ such that both $A\pm\mu\colon\mathcal{M}\rightarrow\mathcal{M}$ have dense range \cite{Ebert}.
	\begin{definition}
		Let $(G,M,E)$ be a $\mathbb{Z}_2$-graded $G$-triple equipped with a regular, self-adjoint operator $B$. Then $B\colon\mathcal{E}^0\rightarrow\mathcal{E}^0$ is said to be \textit{$G$-invertible at infinity} if there exists a non-negative, $G$-invariant, cocompactly supported smooth function $f$ on $M$ such that $B^2+f\in\mathcal{L}(\mathcal{E}^2,\mathcal{E}^0)$ has an inverse $(B^2+f)^{-1}$ in $\mathcal{L}(\mathcal{E}^0,\mathcal{E}^2)$.
	\end{definition}
	\begin{remark}
		When the acting group $G$ is trivial, we will use the term \textit{invertible at infinity}; this is consistent with the usage in \cite{Bunke}. 
	\end{remark}
	\subsection{Equivariant Fredholmness}
	To prove that $G$-invertible-at-infinity operators have an equivariant index, we adapt Bunke's estimates from \cite{Bunke}. The difference in our approach is that the Hilbert $C^*(G)$-module structure of $\mathcal{E}^i$ arises from the $G$-action. Still, we find that most of the estimates in \cite{Bunke} carry over to our setting.
	
	The next lemma is a $G$-equivariant analogue of Lemma 1.4 in \cite{Bunke}.
	\begin{lemma}\label{lem:d}
		$$d\coloneqq\inf_{\psi\in\mathcal{E}^2,\norm{\psi}_{\mathcal{E}^0}=1}\left(\norm{B\psi}^2_{\mathcal{E}^0}+\norm{\sqrt{f}\psi}^2_{\mathcal{E}^0}\right)>0.$$
	\end{lemma}
	\begin{corollary} Let $d$ be as above and $\lambda\in\mathbb{R}$. Then
		$$\norm{(B^2+f+\lambda^2)\psi}_{\mathcal{E}^0} \geq (d+\lambda^2)\norm{\psi}_{\mathcal{E}^0} \qquad \forall \psi\in\mathcal{E}^2.$$
	\end{corollary}
	Let us denote the resolvent by $R(\lambda)\coloneqq(B^2+f+\lambda^2)^{-1}:\mathcal{E}^0\rightarrow\mathcal{E}^2$ whenever it exists. The next lemma is an analogue of Lemma 1.5 in \cite{Bunke}. We give a detailed proof for later reference.
	\begin{lemma}
		\label{resolventlemma}
		Suppose $B\colon\mathcal{E}^1\rightarrow\mathcal{E}^0$ is $G$-invertible at infinity. Then
		\begin{enumerate}[(a)]
			\item for all $\lambda\geq 0$, $R(\lambda)\in\mathcal{L}(\mathcal{E}^0,\mathcal{E}^2)$ exists, and $$\norm{R(\lambda)}_{\mathcal{L}(\mathcal{E}^0)}\leq(d+\lambda^2)^{-1};$$
			\item there exists $C$ such that for all $\lambda\geq 0$,
			$$\norm{B^2 R(\lambda)}_{\mathcal{L}(\mathcal{E}^0)}\leq C.$$
		\end{enumerate}
	\end{lemma}
	\begin{proof}
	Assume that (a) is true for all $0\leq\lambda\leq\Lambda$. That this is true for $\Lambda = 0$ follows from $G$-invertible at infinity and the corollary above. Indeed, since the inclusion $\mathcal{E}^2\hookrightarrow\mathcal{E}^0$ is bounded adjointable, $R(\lambda)\in\mathcal{L}(\mathcal{E}^0)$ for such $\lambda$. To get the estimate, notice that for any $\phi\in\mathcal{E}^0$, the above corollary with $\lambda = 0$ applied to the element $R(0)\phi\in\mathcal{E}^2$ gives 
	$$\norm{R(0)\phi}_{\mathcal{E}^0}\leq\frac{1}{d}\norm{(B^2+f)R(0)\phi}_{\mathcal{E}^0}=\frac{1}{d}\norm{\phi}_{\mathcal{E}^0},$$
	which proves $\norm{R(\lambda)}_{\mathcal{L}(\mathcal{E}^0)}\leq(d+\lambda^2)^{-1}$ for all $\lambda$ in this range. With this in hand, we can show existence of $R(\lambda)$ for $\lambda$ in the range $|\lambda^2-\Lambda^2|<d+\Lambda^2.$ For such $\lambda$ it is true that
	$$\norm{(\Lambda^2-\lambda^2)R(\Lambda)}_{\mathcal{L}(\mathcal{E}^0)}\leq |\Lambda^2-\lambda^2|\norm{R(\Lambda)}_{\mathcal{L}(\mathcal{E}^0)}<(d+\Lambda^2)\norm{R(\Lambda)}_{\mathcal{L}(\mathcal{E}^0)}\leq 1,$$
	where the final inequality follows from $\norm{R(\Lambda)}_{\mathcal{L}(\mathcal{E}^0)}\leq\frac{1}{d+\Lambda^2}.$ Thus the series 
	$$\sum_{i=0}^\infty (\Lambda^2-\lambda^2)^i R(\Lambda)^i$$ converges and defines an element of $\mathcal{L}(\mathcal{E}^0)$ with adjoint $\sum_{i=0}^\infty(\Lambda^2-\lambda^2)^i (R(\Lambda)^*)^i.$ Note that we have $$R(\lambda) = R(\Lambda)\left(\sum_{i=0}^\infty (\Lambda^2-\lambda^2)^i R(\Lambda)^i\right).$$
	Thus for all $\lambda$ such that $|\lambda^2-\Lambda^2|<d+\Lambda^2$, $R(\lambda)\in\mathcal{L}(\mathcal{E}^0,\mathcal{E}^2)$ exists. We can now apply the above corollary to $R(\lambda)\phi\in\mathcal{E}^2$ for any $\phi\in\mathcal{E}^0$, which yields the desired estimate for $\lambda$ in this interval:
	$$\norm{R(\lambda)\phi}_{\mathcal{E}^0}\leq\frac{1}{d+\lambda^2}\norm{(B^2+f+\lambda^2)R(\lambda)\phi}_{\mathcal{E}^0}=\frac{1}{d+\lambda^2}\norm{\phi}_{\mathcal{E}^0}.$$
	Iterating this argument countably many times, we exhaust the positive part of $\mathbb{R}$ and get (a). (b) follows from (a) by the triangle inequality applied to
	$$B^2(B^2+f+\lambda^2)^{-1}=(B^2+f+\lambda^2)(B^2+f+\lambda^2)^{-1} - (f+\lambda^2)(B^2+f+\lambda^2)^{-1},$$
	which shows that, for all $\phi\in\mathcal{E}^0$,
	\begin{align*}
	\norm{B^2 R(\lambda)\phi}_{\mathcal{E}^0}&\leq\norm{\phi}_{\mathcal{E}^0} + \norm{(f+\lambda^2)R(\lambda)\phi}_{\mathcal{E}^0}\leq C\norm{\phi}_{\mathcal{E}^0}.\qedhere
	\end{align*}
	\end{proof}
	\begin{remark}
		\label{rem:complexlambda}
		The above proof also shows that $R(\lambda)$ exists for all $\lambda\in\mathbb{C}$ with $\lambda^2>-d$.
	\end{remark}
	We would like to form the operator $R(0)^\frac{1}{2}$ via functional calculus on $R(\lambda)\in\mathcal{L}(\mathcal{E}^0)$, in order to define a bounded version of $B$. Note that $R(\lambda)$ is a self-adjoint element of the $C^*$-algebra $\mathcal{L}(\mathcal{E}^0)$. 
	
	We have the following two estimates relating to $B$ and $R(\lambda)$, as equivariant analogues of Lemmas 1.6 and 1.7 of \cite{Bunke}.
	\begin{lemma}
		\label{lem:BR}
		We have $BR(\lambda)\in\mathcal{L}(\mathcal{E}^0)$ and
		$$\norm{BR(\lambda)}_{\mathcal{L}(\mathcal{E}^0)} \leq C(d+\lambda^2)^{-1/2}$$
		for some $C<\infty$ independent of $\lambda\geq 0$, with adjoint $$(BR(\lambda))^*=BR(\lambda)+R(\lambda)c(df)R(\lambda),$$
		where $c$ denotes Clifford multiplication.
	\end{lemma}
	\begin{lemma}
		The commutator of $B$ and $R(\lambda)$ acts on $\phi\in\mathcal{E}^1$ by
		$$\left[B,R(\lambda)\right]\phi = -R(\lambda)c(df)R(\lambda)\phi.$$
	\end{lemma}
	\begin{lemma}\label{resolventofresolvent}
		The operator 
		$$\left(R(0)+\kappa\right)^{-1}\in\mathcal{L}(\mathcal{E}^0)$$
		exists for all $\kappa\in (-\infty,-\frac{1}{d})\cup (0,\infty)$.
	\end{lemma}
	\begin{proof}
		By Remark \ref{rem:complexlambda}, that there exists a sufficiently small real number $\mu$ such that $B^2+f+\mu i$ and $B^2+f-\mu i$ are both invertible. By Proposition 4.1 of \cite{KaadLesch}, this means that $B^2+f$ is a regular self-adjoint operator with spectrum contained in $\left(-\infty,d\,\right]$. Upon taking the inverse, the continuous functional calculus for regular self-adjoint operators (Theorem 1.19 in \cite{Ebert}) implies that spectrum of $R(0)\in\mathcal{L}(\mathcal{E}^0)$ is contained in $\left[0,\frac{1}{d}\right]$. It follows that $(R(0)+\kappa)^{-1}\in\mathcal{L}(\mathcal{E}^0)$ exists for all $\kappa\in (-\infty,-\frac{1}{d})\cup (0,\infty)$.
	\end{proof}
	\begin{definition}
		\label{def:definitionofF}
		Suppose  $(G,M,E)$ is a $G$-triple equipped with an operator $B$ that is $G$-invertible at infinity. Then for any $\psi\in\mathcal{E}^1$, the integral
		$$\frac{2}{\pi}\int_{0}^\infty BR(\lambda)\psi\,d\lambda$$
		converges in $\mathcal{E}^0$ and defines a bounded operator $\mathcal{E}^1\rightarrow\mathcal{E}^0$, where elements of $\mathcal{E}^1$ are given the $\mathcal{E}^0$-norm. This operator extends to an odd operator $F\in\mathcal{L}(\mathcal{E}^0)$.
	\end{definition}
		We now proceed as in \cite{Bunke} Lemma 1.8, with $D$ replaced by $B$, $H^i$ replaced by $\mathcal{E}^i$ and $B(H^0)$ replaced by $\mathcal{L}(\mathcal{E}^0)$. One verifies that
		$$R(\lambda)=\frac{1}{\lambda^2}R(0)\left(R(0)+\frac{1}{\lambda^2}\right)^{-1},$$
		where the inverse on the right-hand side exists by Lemma \ref{resolventofresolvent}. A manipulation given in \cite{Bunke} Lemma 1.8 then shows that
		$$\frac{2}{\pi}\int_0^\infty R(\lambda)\,d\lambda = R(0)^{1/2},$$ 
		the right-hand side being defined by functional calculus in $\mathcal{L}(\mathcal{E}^0).$ The operator $R(0)^{1/2}B$ extends by continuity to an element $L\in\mathcal{L}(\mathcal{E}^0)$ such that for $\psi\in\mathcal{E}^1$, $$F\psi=L\psi-\frac{2}{\pi}\int_0^\infty R(\lambda)c(df)R(\lambda)\psi\,d\lambda.$$ 
		The continuous extension of this operator defines $F\in\mathcal{L}(\mathcal{E}^0)$.
	\begin{proposition}
		The above definition of $F$ is equivalent to
		$$F\phi\coloneqq\frac{2}{\pi}\int_0^\infty BR(\lambda)\phi\,d\lambda\qquad\forall\phi\in\mathcal{E}^0.$$
	\end{proposition}
	\begin{proof}
		Let $\phi = \lim_{n\rightarrow\infty}\phi_n$, where $\phi_n\in\mathcal{E}^1$. Then for all $\psi\in\mathcal{E}^1$, we have
		\begin{align*}
		\langle\psi,F\phi\rangle_{\mathcal{E}^0}&=\lim\langle\psi,\frac{2}{\pi}\int_0^\infty BR(\lambda)\phi_n\,d\lambda\rangle_{\mathcal{E}^0}\\
		&=\frac{2}{\pi}\int_0^\infty\langle R(\lambda)B\psi,\lim\phi_n\rangle_{\mathcal{E}^0}\,d\lambda\\
		&=\langle\psi,\frac{2}{\pi}\int_0^\infty BR(\lambda)\phi\,d\lambda\rangle_{\mathcal{E}^0}.\qedhere
		\end{align*}
	\end{proof}
	The next result follows from the proof of Lemma 1.11 in \cite{Bunke} and Theorem \ref{thm:G-Rellich} of the present paper.
	\begin{proposition}
		Let $B$ be $G$-invertible at infinity. Then $F^2\sim 1$ modulo $\mathcal{K}(\mathcal{E}^0)$.
	\end{proposition}
	\begin{proof}
		Note that by \cite{Bunke} Lemmas 1.8 and 1.9, $R(0)^{1/2}B$ extends by continuity to an operator $L\in\mathcal{L}(\mathcal{E}^0)$ and that
		$F$ differs from $L$ by a compact operator.
		Furthermore, $F - F^*\in\mathcal{K}(\mathcal{E}^0)$. Thus it is sufficient to show that $LL^* - 1\in\mathcal{K}(\mathcal{E}^0)$. For $\psi\in\mathcal{E}^2$, we have
		\begin{align*}
		(LL^* - 1)\psi &= R(0)^{1/2}B^2RR(0)^{1/2}(B^2+f)\psi-\psi\\
		&=-R(0)^{1/2}fR(0)^{1/2}\psi.
		\end{align*}
		One calculates that $[f,R(\lambda)]=R(\lambda)(Bc(df)+c(df)B)R(\lambda)\in\mathcal{L}(\mathcal{E}^0).$ Observing Theorem \ref{thm:G-Rellich}, one sees that $fR(0)^{1/2}$ differs from $R(0)^{1/2}f$ by a compact operator.
	\end{proof}
	This allows us to state our first main result:
	\begin{theorem}
		\label{thm:G-Index}
		Let $(G,M,E)$ be a $\mathbb{Z}_2$-graded $G$-triple equipped with an odd operator $B$ that is $G$-invertible at infinity. Then the bounded transform of $B$, $F\in\mathcal{L}(\mathcal{E}^0)$, is $C^*(G)$-Fredholm with an index in $K_0(C^*(G)).$
	\end{theorem}
	\noindent We shall write $\ind_G(F)$ for the $C^*(G)$-index of $F$.
	
	\subsection{A Simplified Definition of $F$}
	\label{subsec:equivalentdef}
	We now show that in fact
	$$F = \frac{2}{\pi}\int_0^\infty BR(\lambda)\,d\lambda = BR(0)^{1/2}.$$ 
	This will simplify certain calculations, for instance in section \ref{sec:PSC}. The argument uses facts about regular operators and Bochner integration.
	
	It follows from the proofs of Lemmas 9.1 and 9.2 in \cite{Lance} that the operator $R(0)^{1/2}$ has range equal to $\mathcal{E}^1$. Moreover, using the functional calculus for regular operators (see \cite{Kustermans} section 7, \cite{Ebert} Theorem 1.19 and \cite{Lance} Chapter 10), we deduce that $R(0)^{1/2}\in\mathcal{L}(\mathcal{E}^0,\mathcal{E}^1)$. Using these facts, we have that:
	
	\begin{proposition}\label{prop:equivalentdefinition}
		Let $F\in\mathcal{L}(\mathcal{E}^0)$ be as in the previous subsection. Then $F=BR(0)^{1/2}$.
	\end{proposition}
	\begin{proof}
		Recall that bounded linear maps commute with Bochner integration (\cite{Border} Lemma 11.45). Since $F$ is the continuous extension of $$\mathcal{E}^1\ni\psi\mapsto\frac{2}{\pi}\int_0^\infty BR(\lambda)\psi\,d\lambda\in\mathcal{E}^0,$$
		$F$ acts on a general element $\phi=\lim_{n\rightarrow\infty}\psi_n\in\mathcal{E}^0$, where $\psi_n\in\mathcal{E}^1$, by 
		$$F\phi = \lim_{n\rightarrow\infty}\frac{2}{\pi}\int_0^\infty BR(\lambda)\psi_n\,d\lambda.$$ Since $R(0)^{1/2}$ is bounded $\mathcal{E}^0\rightarrow\mathcal{E}^1$,
		\begin{align*}
		BR(0)^{1/2}\phi = BR(0)^{1/2}\lim_{n\rightarrow\infty}\psi_n= B\lim_{n\rightarrow\infty}R(0)^{1/2}\psi_n,
		\end{align*}
		where the second limit is taken in $\mathcal{E}^1$. This is equal to
		\begin{align*}
		\lim_{n\rightarrow\infty}BR(0)^{1/2}\psi_n&= \lim_{n\rightarrow\infty}\frac{2B}{\pi}\left(\int_0^\infty R(\lambda)\,d\lambda\right)_0\psi_n,
		\end{align*}
		where the integration is performed in either $\mathcal{L}(\mathcal{E}^0)$ or $\mathcal{L}(\mathcal{E}^1)$. This equals 
		$$\lim_{n\rightarrow\infty}\frac{2B}{\pi}\int_0^\infty R(\lambda)\psi_n\,d\lambda,$$ since pairing with $\psi_n$ is a bounded linear map $\mathcal{L}(\mathcal{E}^1)\rightarrow\mathcal{E}^1$. Since the bounded operator $B\colon\mathcal{E}^1\rightarrow\mathcal{E}^0$ commutes with integration in $\mathcal{E}^1$, this equals
		\begin{align*}BR(0)^{1/2}\phi &= \frac{2}{\pi}\lim_{n\rightarrow\infty}\int_0^\infty BR(\lambda)\psi_n\,d\lambda.\qedhere\end{align*}
	\end{proof}
	This gives another proof of:
	\begin{corollary}
		$F = BR(0)^{1/2}\colon\mathcal{E}^0\rightarrow\mathcal{E}^0$ is an odd operator.
	\end{corollary}
	\begin{proof}
		The functional calculus of an even operator is even, hence $R(0)^{1/2}$ is even. $F$ is the composition of the odd operator $B$ with $R(0)^{1/2}$.
	\end{proof}
	\begin{remark}
		To make sense of the Bochner integrals of $R(\lambda)$ used above in the context of the non-separable Banach spaces $\mathcal{L}(\mathcal{E}^0,\mathcal{E}^1)$ and $\mathcal{L}(\mathcal{E}^1)$, it is necessary for the integrand to be a strongly measurable function of $\lambda\in[0,\infty)$. By Pettis' measurability theorem (\cite{Pettis} Theorem 1.1) and the fact that $\lambda\mapsto R(\lambda)$ is continuous, it suffices to show that the image $R(\lambda)$ for all $\lambda$ is contained in a closed separable subspace of the codomain. Indeed this follows by expanding $R(\lambda)$ as a Neumann series over countably many intervals, as shown in the following lemma, which we state for $\mathcal{L}(\mathcal{E}^0,\mathcal{E}^1)$, but also holds for $\mathcal{L}(\mathcal{E}^1)$.
		\begin{lemma}
			The image of the map $[0,\infty)\rightarrow\mathcal{L}(\mathcal{E}^0,\mathcal{E}^1),$ $\lambda\mapsto R(\lambda)$ lies in a separable subspace of $\mathcal{L}(\mathcal{E}^0,\mathcal{E}^1)$.
		\end{lemma}
		\begin{proof}
			Define a sequence $(\Lambda_k)_{k\in\mathbb{N}}$ by $\Lambda_k\coloneqq\sqrt{\frac{(k-1)d}{2}}$, with $d$ as in Lemma \ref{lem:d}. From the proof of Lemma \ref{resolventlemma} one sees that, for $k\geq 1$ and $\lambda\in [a_k,a_{k+1}]\eqqcolon J_k$, $$R(\lambda)=R(\Lambda_k)\left(\sum_{i=0}^\infty (\Lambda_k^2-\lambda^2)^i R(\Lambda_k)^i\right).$$ 
			Thus for $\lambda$ belonging to each of the countably many intervals $J_k$, $k\in\mathbb{N}$, the resolvent $R(\lambda)$ belongs to the closure of the span of $\mathcal{A}_k\coloneqq\{R(\Lambda_k)^i\colon i\geq 0\}$. It follows that the closure of the span of $\bigcup_{k\in\mathbb{N}}\mathcal{A}_k,$ contains  $R(\lambda)$ for all $\lambda\in [0,\infty)$.
		\end{proof}
	\end{remark}
	\subsection{$G$-invertible Operators as $KK$-elements}
	\label{subsec:GInvertibleKK}
	Given the results of the previous section, we can now complete the proof of:
\begin{theorem}
	\label{thm:CalliasKKClass}
	Let $(G,M,E)$ be a $G$-triple equipped with an operator $B$ that is $G$-invertible at infinity. Let $F$ be the bounded transform of $B$ defined using a cocompactly supported function $f$, as in Definition \ref{def:definitionofF}. Then $(\mathcal{E}^0,F)$ is a Kasparov module over the pair of $C^*$-algebras $(\mathbb{C},C^*(G))$. The class $[\mathcal{E}^0,F]\in KK(\mathbb{C},C^*(G))$ is independent of the choice of the function $f$.
\end{theorem}
\begin{proof}
In view of the previous results, it remains to prove the last assertion. This follows from Theorem \ref{thm:G-Rellich} and the computation in the proof of \cite{Bunke} Lemma 1.10.
\end{proof}
%

	The image of $[\mathcal{E}^0,F]$ under the isomorphism (see \cite{Blackadar} 17.5.5) $$KK(\mathbb{C},C^*(G))\cong K_0(C^*(G))$$ coincides with $\ind_G(F)$ defined after Theorem \ref{thm:G-Index}, hence we will denote this map also by $\ind_G$. 
\begin{remark}
	With $F$ as above, one can in fact show that $[F,h]\in\mathcal{K}(\mathcal{E}^0)$ for all $h$ in $C_g^G(M)$, where $C_g^G(M)$ is the space of continuous functions on the Higson $G$-compactification of $M$ (see section \ref{sec:Phi}). Indeed, it suffices to establish this for all $h\in C_g^{\infty,G}(M)$. One can, for example, proceed as in \cite{Bunke} Lemma 1.12, replacing $C_g^{\infty}(M)$ with $C_g^{\infty,G}(M)$. However, we give a simpler proof is as follows. First note that $hR(0)^{1/2}$ and $R(0)^{1/2}h$ differ by the operator
	$$\frac{2}{\pi}\int_0^\infty R(\lambda)(Bc(d h)+c(dh)B)R(\lambda)\,d\lambda.$$
	Since $h\in C_g^{\infty,G}(M)$, $\norm{dh}_{T^*M}\in C_0^{\infty,G}(M)$. Thus $dh$ can be approximated by cocompactly supported endomorphisms, for which the integral converges absolutely. On the other hand, $hB$ and $Bh$ differ by $c(dh)$, which is again a limit of cocompactly supported endomorphisms. Thus $hBR(0)^{1/2}-BR(0)^{1/2}h$ is compact.
\end{remark}
	\vspace{1cm}

	\section{$G$-Callias-type Operators}
\label{sec:G-Callias}

\noindent We now define, for $G$ a general Lie group, {\it $G$-invariant Callias-type operators} and prove that they are $G$-invertible at infinity. This notion generalises the Callias-type operators studied in \cite{Bunke} section 2 to the equivariant, non-cocompact setting.

\begin{definition}
	\label{def:Gadmissible}
	Let $E\rightarrow M$ be a $\mathbb{Z}_2$-graded $G$-Clifford bundle with Dirac operator $D$. An odd-graded, $G$-invariant endomorphism $\Phi\in C^1(M,\End{E})$ is called {\it $G$-admissible for $D$} (or simply {\it $G$-admissible}) if
	\begin{enumerate}[(a)]
		\item $\Phi D + D\Phi$ is a bounded, order-0 bundle endomorphism; 
		\item $\Phi$ is self-adjoint with respect to the inner product on $E$;
		\item there exists a cocompact subset $K\subseteq M$ and $C>0$ such that 
		$$\Phi D + D\Phi + \Phi^2\geq C \textnormal{ on } M\backslash K.$$
	\end{enumerate}
\end{definition}
For example, let $N$ be a cocompact $G$-equivariantly spin manifold and $D$ the spin-Dirac operator. Let $\chi:\mathbb{R}\rightarrow\mathbb{R}$ be the identity function, and extend it naturally to a function $\tilde{\chi}$ on $N\times\mathbb{R}$. Then multiplication by $\tilde{\chi}$ is an endomorphism on the spinor bundle that is $G$-admissible for $D$.

In subsection \ref{subsec:ConstructingPhi}, we will use an equivariant version of the Higson corona of $M$ to construct more examples of $G$-admissible endomorphisms.

\begin{definition}
	Let $E$ and $D$ be as in Definition \ref{def:Gadmissible}. A {\it $G$-invariant Callias-type operator} (or simply a {\it $G$-Callias-type operator}) is an operator of the form $$B\coloneqq D+\Phi,$$
	where the endomorphism $\Phi$ is $G$-admissible for $D$.
\end{definition}

We first show that such an operator is $G$-invertible at infinity and hence $C^*(G)$-Fredholm. 
\subsection{Positivity of $B^2+f$}
Let $B=D+\Phi$ be a $G$-Callias-type operator. We first show that there exists a cocompactly supported, $G$-invariant function $f$ such that $B^2+f$ is a positive unbounded operator with respect to the $C^*(G)$-valued inner product on $\mathcal{E}^0$.
\begin{lemma}
	Let $B = D+\Phi$ be a $G$-Callias-type operator on $E\rightarrow M$. Then there exist a $G$-invariant, cocompactly supported function $f$ and a constant $C>0$ such that for all $s\in H^2(E)$,
	$$\langle(B^2+f)s,s\rangle_{H^0}\geq C\langle s,s\rangle_{H^0}.$$
\end{lemma}
\begin{proof}
	Let $\pi\colon M\rightarrow M/G$ be the projection. The $G$-bundle $E$ descends to a topological vector bundle $\check{E}$ over $M/G$, while the $G$-invariant bundle map $\Phi D + D\Phi + \Phi^2$ descends to a continuous bundle map $\chi$ on $\check{E}$. Let $K$ be the cocompact subset in Definition \ref{def:Gadmissible} (c). Then $\Phi D + D\Phi + \Phi^2$ is bounded below by the same constant as for $\chi$ over $\pi(K)$, namely
	$$\inf_{x\in \pi(K)}\left(\inf_{v\in \check{E}_x}\left(\frac{\langle \chi v,v\rangle}{\norm{v}^2}\right)\right) \geq \inf_{x\in \pi(K)}(-\norm{\chi}).$$
	
	\noindent Adding a sufficiently large, compactly supported function $\check{f}\colon M/G\rightarrow\mathbb[0,\infty)$ to $\Phi D+D\Phi+\Phi^2$ makes it positive on $K$. The result now follows by taking $f$ to be the pullback of $\check{f}$ to $M$, observing that on $M\backslash K$ we have $\Phi D + D\Phi + \Phi^2\geq C$, where $C$ is the constant in Definition \ref{def:Gadmissible} (c).
\end{proof}
The above lemma and Corollary \ref{prop:Bibounded} imply that, given a $G$-Callias-type operator $B$, $B^2+f$ defines an element of $\mathcal{L}(\mathcal{E}^2,\mathcal{E}^0)$. We now show that this operator is positive in the sense of $C^*(G)$. The proof is inspired by Kasparov's proof of \cite{Kasparov} Lemma 5.3.
\begin{proposition}
	\label{prop:Positivity}
	Let $(G,M,E)$ be a $\mathbb{Z}_2$-graded $G$-triple and $B$ a $G$-Callias-type operator on $E$. Then there exists a $G$-invariant cocompactly supported function $f$ and a constant $C>0$ such that for all $s\in\mathcal{E}^2$,
	$$\langle(B^2+f)s,s\rangle_{\mathcal{E}^0}\geq C\langle s,s\rangle_{\mathcal{E}^0}.$$
\end{proposition}
\begin{proof}
	Let $B=D+\Phi$ and $\mathfrak{c}$ be a cut-off function on $M$. Since $\langle D^2 e,e\rangle_{\mathcal{E}^0}\geq 0$ for all $e\in C_c^\infty(E)$, it suffices to show that there exist $f$ and $C>0$ such that
	$$\langle (B^2+f-D^2)e,e\rangle_{\mathcal{E}^0} = \langle (D\Phi+\Phi D+\Phi^2+f)e,e\rangle_{\mathcal{E}^0}\geq C\langle e,e\rangle_{\mathcal{E}^0}.$$
	In other words, we want to show
	$$\langle (D\Phi+\Phi D+\Phi^2+f-C)e,e\rangle_{\mathcal{E}^0}\in C^*(G)_+.$$ Choose $f$ as in the previous lemma, so that $D\Phi+\Phi D+\Phi^2+f-C$ is a bounded positive operator on $L^2(E)$. Since $G$ acts by unitaries on $L^2(E)$, and conjugating by unitaries preserves the functional calculus, $D\Phi+\Phi D+\Phi^2+f-C$ has a bounded, $G$-invariant positive square root $Q$. Now the operator $Q\mathfrak{c}Q$ has cocompactly compactly supported Schwartz kernel 
	$$k_{Q\mathfrak{c}Q}(x,y)=\int_M k_Q(x,z)\mathfrak{c}(z)k_Q(z,y)\,d\mu(z),$$ 
	since $\mathfrak{c}$ has cocompactly compact support. Let $\tilde{K}$ be a cocompactly compact subset of $M$ such that supp$(k_{Q\mathfrak{c}Q})\subseteq\tilde{K}\times\tilde{K}$. Define $a\colon G\rightarrow\mathbb{\mathbb{R}}$ to be the function taking $g\in G$ to:
	$$\langle Q\mathfrak{c}Qg(e),g(e)\rangle_{H^0(E)}=\int_M\bigg\langle\int_M k_{Q\mathfrak{c}Q}(x,y)g(e)(y)\,d\mu(y),g(e)(x)\bigg\rangle_E d\mu(x).$$
	Then $\textnormal{supp}(a)$ is contained in
	$$\{g\in G\,|\,\textnormal{supp}(g(e))\cap\tilde{K}\neq\emptyset\}\subseteq\{g\in G\,|\,\textnormal{supp}(g(e))\cap G\cdot\textnormal{supp}(e)\cap\tilde{K}\neq\emptyset\}.$$
	Since the set $G\cdot\textnormal{supp}(e)\cap\tilde{K}$ is compact, $\textnormal{supp}(a)$ is a compact subset of $G$, by properness of the $G$-action. Thus the map $G\mapsto H^0(E)$, $g\mapsto\sqrt{\mathfrak{c}}Q(g(e))$ has compact support in $G$. It follows that for any unitary representation of $G$ on a Hilbert space $(H,(\,,\,)_H)$ and $h\in H$,
	$$v\coloneqq \int_G \mu^{-1/2}(g)\sqrt{\mathfrak{c}}Q(g(e))\otimes g(h)\,dg$$
	is a well-defined vector in $H^0(E)\otimes H$. Its norm $\norm{v}_{H^0(E)\otimes H}$ is equal to
	\begin{align*}
	&\int_G\int_G\mu^{-1/2}(g')\mu^{-1/2}(g)\langle\sqrt{\mathfrak{c}}Qg(e),\sqrt{\mathfrak{c}}Q(g'(e))\rangle_{H^0}\cdot(g(h),g'(h))_H\,dg\,dg'\\
	&=\int_G\int_G\mu^{-1/2}(g)\mu^{-1/2}(g')\langle Q\mathfrak{c}Qg(e),g'(e)\rangle_{H^0}\cdot(g(h),g'(h))_H\,dg'\,dg.
	\end{align*}
	Since $\langle\,,\,\rangle_{H^0(E)}$ and $(\,,\,)_H$ are $G$-invariant (see Remark \ref{rem:unitary}), this equals
	$$\int_G\int_G\mu^{-1/2}(g)\mu^{-1/2}(g')\langle g'^{-1}(Q\mathfrak{c}Q)g(e),e\rangle_{H^0}\cdot(g'^{-1}g(h),h)_H\,dg'\,dg.$$
	$$=\int_G\int_G\mu^{-1/2}(g)\mu^{-1/2}(g')\langle g^{-1}(Q\mathfrak{c}Q)g'^{-1}g(e),e\rangle_{H^0}\cdot(g'^{-1}g(h),h)_H\,dg'\,dg.$$
	Substituting $g'\mapsto gg'$ and following simliar computations to the proof of \cite{Kasparov} Lemma 5.3, one sees that this is equal to		$$\int_G\bigg\langle\left(\int_G g(Q\mathfrak{c}Q)\,dg\right)(e),e\bigg\rangle_{\mathcal{E}^0}(g')\cdot(g'(h),h)_H\,dg'.$$
	Thus $\langle(\int_G g(Q\mathfrak{c}Q)\,dg)(e),e\rangle_{\mathcal{E}^0}$ is a positive operator on $H$ for all unitary representations of $G$, where we let $f\in C_c^\infty(G)$ act on $H$ by 
	$$f\cdot h\coloneqq \int_G f(g)g(h)\,dg.$$
	It follows that the element
	$$\langle (D\Phi+\Phi D+\Phi^2+f-C)e,e\rangle_{\mathcal{E}^0}=\langle Q^2 e,e\rangle_{\mathcal{E}^0}=\bigg\langle\bigg(\int_G g(Q\mathfrak{c}Q)\,dg\bigg)(e),e\bigg\rangle_{\mathcal{E}^0}$$
	is in $C^*(G)_+$, where we have used the fact that, since $Q$ is $G$-invariant,
	\begin{align*}\int_G g(Q\mathfrak{c}Q)\,dg &= Q^2.\qedhere\end{align*}
\end{proof}
Next, we show that $B$ has a regular self-adjoint extension. In fact, we show that $B$ is essentially self-adjoint with regular closure.
\subsection{Essential Self-adjointness and Regularity of $B$}
\label{subsec:essentialregularity}
The notation in this subsection and the next will distinguish between the formally self-adjoint operator $B$ and its closure $\overline{B}$.
\begin{proposition}
	\label{prop:extensiontononcocompact}
	Let $(G,M,E)$ be a $\mathbb{Z}_2$-graded $G$-triple with $M$ complete. Let $B$ be a $G$-Callias-type operator on $E$. Then $B$ is essentially self-adjoint in $\langle\,,\,\rangle_{\mathcal{E}^0}$, and $\overline{B}$ is regular.
\end{proposition}
\begin{proof}
	When $(G,M,E)$ is cocompact with $M$ complete, the regularity and self-adjointness of $\overline{B}$ was established in \cite{Kasparov} Theorem 5.8. Now suppose $M$ is complete and non-cocompact. There exists a family $\{a_\epsilon\colon\epsilon > 0\}$ of compactly supported smooth functions taking values in $[0,1]$ satisfying:
	\begin{enumerate}[(i)]
	\item $\displaystyle\bigcup_{\epsilon>0}\{a_\epsilon^{-1}(1)\}=M;$
	\item $\displaystyle\sup_{x\in M}\norm{da_\epsilon (x)}\leq\epsilon;$
	\item $\displaystyle a_{\epsilon_1}^{-1}(1)\subseteq a_{\epsilon_2}^{-1}(1)\textnormal{ if }\epsilon_2\leq\epsilon_1.$
	\end{enumerate}

	Let $s\in C_c^\infty(E)$. Choose $\epsilon$ so that $a_\epsilon\equiv 1$ on supp$(s)$. Take an exhaustion of $M$ by cocompact, $G$-stable open subsets $\{U_k\colon k\in\mathbb{N}\}$ as in the proof of Proposition \ref{prop:Bbounded}, and pick $i$ large enough so that $U_i$ contains supp$(a_\epsilon)$. Now form the double $M'$ of the closure $\overline{U_i}$, using a $G$-equivariant collar neighbourhood (which exists by \cite{Kankaanrinta} Theorem 3.5), and extend the action of $G$ to $M'$ naturally. Then $M'$ is a cocompact $G$-Riemannian manifold without boundary, and all $G$-structures on $U_i$ naturally extend to $M'$ also.
	
	In particular, this gives a $G$-Callias-type operator $B'$, acting on $E'\rightarrow M'$. Since $M'$ is cocompact, the closure $\overline{B'}$ is regular and self-adjoint. Now use $\overline{B'}$ to form $G$-Sobolev modules $\{\mathcal{E}'^{,i}\}$ on $M'$, following section \ref{sec:G-Sobolev}. Because $\overline{B'}+ i\colon\mathcal{E}'^{,1}\rightarrow\mathcal{E}'^{,0}$ is onto, we can find $e\in\mathcal{E}'^{,1}$ for which $(\overline{B'}+i)e=s$. We have $\langle s,s\rangle_{\mathcal{E}^{,0}}\geq\langle e,e\rangle_{\mathcal{E}^{,0}}\in C^*(G)_+$ and
	\begin{align*}(\overline{B}+i)(a_\epsilon e)&=(\overline{B'}+i)(a_\epsilon x)\\
	&=[\overline{B'},a_\epsilon]e+a_\epsilon(\overline{B'}+i)e\\
	&=c(\nabla a_\epsilon)e+a_\epsilon(\overline{B'}+i)e.
	\end{align*}
	Note that $c(\nabla a_\epsilon)$ is a bounded operator on $\mathcal{E}^0$, but since $a_\epsilon$ was not assumed to be $G$-invariant, it could fail to be adjointable. Nevertheless, boundedness is enough, since it enables us to conclude 
	$$\norm{(\overline{B}+i)(a_\epsilon e)}_{\mathcal{E}^0} = \norm{c(\nabla a_\epsilon)e+a_\epsilon(\overline{B'}+i)e}_{\mathcal{E}^0}\leq(1+\epsilon)\norm{s}_{\mathcal{E}^0}.$$
	Taking a sequence $\epsilon_i \rightarrow 0$, this implies that $$(\overline{B}+i)(a_{\epsilon_i}e)\rightarrow a_{\epsilon_i}(\overline{B'}+i)e=a_{\epsilon_i} s = s,$$
	thus $\overline{B}+i$ has dense range.
\end{proof}
\subsection{$G$-invertibility at Infinity of $B$}
In this subsection we establish the key result:
\begin{theorem}
	\label{thm:CalliasInvertibility}
	Let $(G,M,E)$ be a $\mathbb{Z}_2$-graded $G$-triple with $B=D+\Phi$ a $G$-Callias-type operator. Then $B$ is $G$-invertible at infinity.
\end{theorem}
\begin{proposition}
	For $0\neq\mu\in\mathbb{R}$, $B^2+\mu^2\colon\mathcal{E}^2\rightarrow\mathcal{E}^0$ is bijective, with an inverse in $\mathcal{L}(\mathcal{E}^0,\mathcal{E}^2)$.
\end{proposition}
\begin{proof}
	By part (e) of \cite{Cecchini} 3.3, surjectivity follows from the fact that $B$ has regular closure, which was shown in the previous subsection. For all $v\in\mathcal{E}^2$, we have $\norm{(B^2+\mu^2)v}_{\mathcal{E}^0}\geq\mu^2\norm{v}_{\mathcal{E}^0}.$ This implies injectivity. By the open mapping theorem, the inverse $(B^2+\mu^2)^{-1}$ is bounded. Adjointability is guaranteed, since from general theory one knows that a bounded inverse of an invertible bounded adjointable operator $T$ between Hilbert $A$-modules $\mathcal{M},$ $\mathcal{N}$ is must be adjointable. Indeed, it follows from \cite{BlackadarCStar} Theorem II.7.2.9 that $\textnormal{ran}(T^*)=\textnormal{ker}(T)^{\perp}=\{0\}=\mathcal{M}$ and $\textnormal{ker}(T^*)=\textnormal{ran}(T)^{\perp}=F^\perp=\{0\}.$ Hence $T^*$ has an inverse satisfying
	\begin{align*}
	\langle T^{-1}y,x\rangle_{\mathcal{M}}=\langle T^{-1}y,T^*(T^*)^{-1}x\rangle_{\mathcal{M}}=\langle y,(T^*)^{-1}x\rangle_{\mathcal{N}}.
	\end{align*}
	Thus the inverse of $T^{-1}$ is $(T^*)^{-1}$.
\end{proof}
\begin{remark}
	The formula $((B^2+\mu^2)^{-1})^*=B^2+(1-\mu^2)(B^2+\mu)^{-1}B^2+(B^2+\mu)^{-1}$ can be proved as in \cite{Cecchini} Proposition 4.9.
\end{remark}
\begin{proof}[Proof of Theorem~\ref{thm:CalliasInvertibility}]
	The analogue of the argument in \cite{Cecchini} subsection 4.10 also works in our situation, thus we only sketch the argument. First one shows that for $\mathbb{R}\ni\mu\neq 0$ and $f\colon M\rightarrow\mathbb{R}$ a $G$-invariant uniformly bounded smooth function with $\mu^2>\norm{f}_\infty$, the operator $B^2+\mu^2+f$ has an inverse in $\mathcal{L}(\mathcal{E}^0)\cap\mathcal{L}(\mathcal{E}^0,\mathcal{E}^2)$. This inverse is given by the Neumann series
	$$(B^2+\mu^2+f)^{-1}=(B^2+\mu^2)^{-1}\sum_{k=0}^\infty(-1)^{k}\left\{f(B^2+\mu^2)^{-1}\right\}^k.$$
	By Proposition \ref{prop:Positivity}, we can pick a $G$-invariant cocompactly supported function $f$ and a constant $C>0$ such that for all $s\in\mathcal{E}^2$, $\langle(B^2+f)s,s\rangle_{\mathcal{E}^0}\geq C\langle s,s\rangle_{\mathcal{E}^0}.$ Pick $\mu\neq 0$ such that $\mu^2>\norm{f}_{\infty}$. We have $$B^2+f=(1-\mu^2(B^2+\mu^2+f)^{-1})(B^2+\mu^2+f).$$ 
	The Cauchy-Schwartz inequality then yields $$\norm{\mu^2(B^2+\mu^2+f)^{-1}}_{\mathcal{L}(\mathcal{E}^0)}\leq\frac{\mu^2}{\mu^2+C}<1,$$ so that the Neumann series
	$$(B^2+f)^{-1}=(B^2+\mu^2+f)^{-1}\sum_{k=0}^{\infty}\left\{\mu^2(B^2+\mu^2+f)^{-1}\right\}^k$$ converges in norm to an element of $\mathcal{L}(\mathcal{E}^0)\cap\mathcal{L}(\mathcal{E}^0,\mathcal{E}^2).$
\end{proof}
	\vspace{1cm}

	\section{The Endomorphism $\Phi$ and the Higson Corona}
\label{sec:Phi}
We now construct $G$-admissible endomorphisms $\Phi$ using the $K$-theory of an equivariant version of the Higson corona used in \cite{Bunke}.

Let $C_b^G(M)$ denote the algebra of bounded, continuous $G$-invariant functions on $M$. Let $C_g^{\infty,G}(M)\subset C_b^G(M)$ denote the subalgebra of smooth functions $f$ such that for all $\epsilon > 0$, there exists a cocompact subset $M_0\subseteq M$ such that 
$$\norm{df}_{T^*M}<\epsilon$$ 
on $M\backslash M_0$. Let $C_0^G(M)$ be the algebra of $G$-invariant functions $f$ on $M$ such that for all $\epsilon > 0$ there exists a cocompact subset $M_1\subseteq M$ such that $|f|<\epsilon$ on $M\backslash M_1$. Thus  $C_g^{\infty,G}(M)$ consists of those bounded smooth $G$-invariant functions $f$ for which $\norm{df}_{T^*M}\in C_0^G(M)$. 
\begin{remark}
	In \cite{Bunke}, the notation $C_g(M)$ is used for the closure in $\norm{\,\cdot\,}_\infty$ of the smooth functions $f$ on $M$ for which $\norm{df}_{T^*M}\in C_0(M)$.
\end{remark}
\begin{definition}
	Let the algebra $C_g^{G}(M)$ be the completion of $C_g^{\infty,G}(M)$ in $C_b^G(M)$ with respect to $\norm{\,\cdot\,}_\infty$. Define 
	$$C(\partial_h^{G}(\overline{M}))= C_g^G(M)/C_0^G(M).$$ 
	The maximal ideal space $\overline{M}^G$ of $C_g^G(M)$ is called the {\it Higson $G$-compactification} of $M$. The maximal ideal space $\partial_h^G\overline{M}$ of $C(\partial_h^G\overline{M})$ is called the {\it Higson $G$-corona} of $M$. 
\end{definition}

The Higson $G$-compactification for $G=\{e\}$ was studied in \cite{Bunke}. Note that there, the notation $\partial_h M$ was used where we have used $\partial_h\overline{M}$.
\begin{lemma}
	$\overline{M}^G$ is a compactification of $M/G$.
\end{lemma}
\begin{proof}
	$C_g^G(M)$ contains $C_0^G(M)\cong C_0(M/G)$, which is the closure of the space of $G$-invariant cocompactly supported functions on $M$ under $\norm{\,\cdot\,}_\infty$. Thus $C_g^G(M)$ separates points of $M/G$ from closed subsets and contains the constant functions. Thus $\overline{M}^G$ is a compactification of $M/G$.
\end{proof}
$\overline{M}^G$ is the unique compactification of $M/G$ such that a function $f\colon M/G\rightarrow\mathbb{C}$ extends continuously to $\overline{M}^G$ if and only if $f\in C_g^G(M)$. 

We will show that the $K$-theory of the Higson $G$-corona is highly non-trivial, which provides motivation for the study of $G$-Callias-type index theory. First, we give a construction of $G$-admissible endomorphisms using the $K$-theory of $\overline{M}^G$.
\subsection{Constructing $\Phi$}
\label{subsec:ConstructingPhi}
Suppose first that $M$ is an odd-dimensional complete non-cocompact $G$-Riemannian manifold. Let $E_0\rightarrow M$ be a complex Clifford bundle with an ungraded Dirac operator $D_{E_0}$. Let $P_0$ be a projection in $\textnormal{Mat}_l(C(\partial_h^G(\overline{M})))$ for some $l\geq 0$. Then $P_0$ lifts to an element $P\in\textnormal{Mat}_l(C_g^{\infty,G}(M))$ such that $P=P^*$ and $P^2-P\in\textnormal{Mat}_l(C_0^{\infty,G}(M))$. Form the $\mathbb{Z}_2$-graded Clifford bundle $E\coloneqq E_0\otimes(\mathbb{C}^l\oplus(\mathbb{C}^l)^{op})$, with the associated Dirac operator 
$$D \coloneqq  \left[
\begin{array}{cc}
0 & D_{E_0}\otimes 1\\
D_{E_0}\otimes 1& 0
\end{array}\right],$$
where $D_{E_0}$ denotes $D_{E_0}$ twisted by the trivial connection $d^l$ on $\mathbb{C}^l\rightarrow M$. Define the endomorphism
$$\Phi \coloneqq  i\otimes\left[
\begin{array}{cc}
0 & 1-2P \\
2P-1 & 0
\end{array}\right]\in C^1(M,\End(E)).$$
\begin{proposition}
	$\Phi$ is a $G$-admissible endomorphism.
\end{proposition}
\begin{proof}
	Clearly $\Phi$ is odd-graded and self-adjoint. Let $\nabla^{E_0}$ denote the Clifford connection on $E_0$. Then the twisted connection on $E_0\otimes\mathbb{C}^l$ is $\nabla^{E_0}\otimes 1+1\otimes d^l$. Since $1-2P$ acts only on the factor $\mathbb{C}^l$, it commutes with $\nabla^{E_0}\otimes 1$. Also, $[1\otimes d^l,1\otimes(1-2P)]=d^{l^2}P,$ where $d^{l^2}$ is the trivial connection on $\End(\mathbb{C}^l)$. Thus we have
	$$D\Phi+\Phi D = 2i\otimes\left[
	\begin{array}{cc}
	-c(d^{l^2}P) & 0\\
	0 & c(d^{l^2}P)
	\end{array}\right].$$
	By definition of $C_g^{G,\infty}(M)$, $c(d^{l^2}P)\rightarrow 0$ and $\Phi^2\rightarrow 1$ at infinity in the direction transverse to the $G$-orbits. Thus there exists a cocompact subset $K\subseteq M$ such that on $M\backslash K$, $D\Phi+\Phi D+\Phi^2\geq c$ for some constant $c>0$.
\end{proof}
By Theorem \ref{thm:CalliasInvertibility}, the operator $B\coloneqq D+\Phi$ is $G$-invertible at infinity. Thus $B$ has an index in $K_0(C^*(G))$ by Theorem \ref{thm:G-Index}. Let $F$ be the bounded transform of $B$. By the same reasoning as in \cite{Bunke} section 2, $\ind_G(F)$ depends only on the class of projection $P_0$ in $K_0(C(\partial_h^G(\overline{M}))$.

Note that the endomorphism $\Phi$ arising from the zero element of $K_0(C(\partial_h^G(\overline{M})))$ gives rise to the invertible operator
$$B=\left[
\begin{array}{cc}
0 & D+i \\
D-i & 0
\end{array}\right],$$
for which $\ind_G(BR(0)^{1/2}) = 0\in K_0(C^*(G)).$

Suppose instead that $M$ is an even-dimensional, complete non-cocompact $G$-Riemannian manifold. Let $E_0\rightarrow M$ be a $\mathbb{Z}_2$-graded Clifford bundle with grading $z$. An element $[U_0]\in K_1(C(\partial_h^G(\overline{M})))$ is represented by a unitary $U_0\in\textnormal{Mat}_l(C(\partial_h^G(\overline{M})))$ for some $l\geq 0$. Let $U$ be a lift of $U_0$ to $\textnormal{Mat}_l(C_g^{\infty,G}(M))$. Form the $\mathbb{Z}_2$-graded bundle $E\coloneqq E_0\otimes(\mathbb{C}^l\oplus(\mathbb{C}^{l,op}))$ with Dirac operator $D$ formed from the connection on $E_0$ twisted by $d^l$. Let
$$\Phi \coloneqq  z\otimes\left[
\begin{array}{cc}
0 & U^* \\
U & 0
\end{array}\right]\in C^1(M,\End(E)),$$
and define $B=D+\Phi$.
\begin{proposition}
	$\Phi$ is a $G$-admissible endomorphism.
\end{proposition}
\begin{proof}
	Clearly $\Phi$ is odd-graded and self-adjoint. Similarly to the calculation in the previous proposition, one has
	$$D\Phi+\Phi D = z\left[
	\begin{array}{cc}
	0 & -c(d^{l^2}U^*)\\
	-c(d^{l^2}U) & 0 
	\end{array}\right].$$
	By definition of $C_g^{G,\infty}(M)$, $c(d^{l^2}(U))\rightarrow 0$ and $\Phi^2\rightarrow 1$ at infinity in the direction transverse to the $G$-orbits. It follows that there is some cocompact subset $K\subseteq M$ such that on $M\backslash K$ $D\Phi+\Phi D+\Phi^2\geq c > 0$ for some $c$.
\end{proof}
As in the odd-dimensional case, one can verify that the $\ind_G(F)$, where $F$ is the bounded transform of $B$, depends only on $[U_0]\in K_1(C(\partial_h^G(\overline{M})))$.

For $\dim M \equiv i\,\, (\textnormal{mod } 2)$, $i = 0,1$, the map
$$K_i(C(\partial_h^G(\overline{M})))\rightarrow K_0(C^*(G)),$$
$$[R]\mapsto\ind_G(F_R),$$
is a homomorphism of abelian groups. Here $R$ denotes a projection or unitary matrix representative of $K_0$ or $K_1$, depending on $\dim M$, and $F_R$ denotes the bounded transform of $B=D+\Phi_R$, where $\Phi_R$ is formed using $R$.

\subsection{Interpretation as a $KK$-product}
\label{subsec:KK}
Let $F$ be the bounded transform of a $G$-Callias-type operator $B=D+\Phi$. In this subsection we interpret the cycle $[\mathcal{E}^0,F]$ in terms of a $KK$-pairing, similar to those constructed in \cite{Bunke} and \cite{Kucerovsky}.

Suppose first that $M$ is even-dimensional, with a Dirac operator $D_0$ acting on a $\mathbb{Z}_2$-graded $G$-Clifford bundle $E_0\rightarrow M$. Given a $G$-admissible $\Phi$ of the kind constructed in the previous subsection, one can form the bundle $E$ and the operator $D$. Let $B=D+\Phi$. Using the procedure in section \ref{sec:G-Sobolev}, form the $G$-Sobolev modules $\mathcal{E}^i_{D_0}$ and $\mathcal{E}^i$ associated to $D_0$ and $B$ respectively. Note that $\mathcal{E}^0 = \mathcal{E}^0_{D_0}\otimes\mathbb{C}^2$.

Let $R_{D_0}(0)\coloneqq (D_0^2+1)^{-1}.$ Then by the same kind of analysis as in section \ref{sec:G-Invertible}, one knows that $R_{D_0}(0)^{1/2}$ is a bounded adjointable operator $\mathcal{E}_{D_0}^0\rightarrow\mathcal{E}_{D_0}^1$. Define its bounded transform $F'\coloneqq D_0 R_{D_0}(0)^{1/2}.$
\begin{proposition}
	The pair $(F',\mathcal{E}^0_{D_0})$ defines a cycle $$[D_0]\coloneqq [F',\mathcal{E}^0_{D_0}]\in KK(C_0^G(M),C^*(G)).$$
\end{proposition}
\begin{proof}
	Each $a\in C_0^G(M)$ defines a bounded adjointable operator $\mathcal{E}_{D_0}^0\rightarrow\mathcal{E}_{D_0}^0$. Let $C_c^G(M)\subseteq C_b(M)$ be the subring of $G$-invariant, cocompactly supported functions on $M$, with closure $C_0^G(M)$ (see subsection \ref{subsec:GInvertibleKK}). It suffices to show that for all $a\in C_c^G(M)$, we have
	$$a((F')^2-1)\in\mathcal{K}(\mathcal{E}^0_{D_0}),\qquad F'a-aF'\in\mathcal{K}(\mathcal{E}^0_{D_0}).$$
	The second relation follows from the same sort of argument used to prove Theorem \ref{thm:CalliasKKClass}, applied to $F'$ instead of $F$. By functional calculus of the regular operator $D_0$ one has
	$$(F')^2-1=-R_{D_0}(0)\in\mathcal{L}(\mathcal{E}^0,\mathcal{E}^2).$$
	Hence by Theorem \ref{thm:G-Rellich}, $a((F')^2-1))$ is a compact operator $\mathcal{E}^0_{D_0}\rightarrow\mathcal{E}^0_{D_0}$.
\end{proof}
From here, our construction is similar that in \cite{Bunke} 2.3.3 and 2.4.3 for the non-equivariant case, so we will be rather brief. The differences are: instead of the algebra $C_g(M)$ used in \cite{Bunke}, we use $\mathbb{C}$; instead of $C_0(M)$ we use $C_0^G(M)$; and the $\mathbb{C}$-Fredholm cycle $[h,F_2]$ used there is replaced by the $C^*(G)$-Fredholm cycle $[D]=[\mathcal{E}^0_{D_0},F']$. 

Define the $\mathbb{Z}_2$-graded Hilbert $C_0^G(M)$-module
$$L\coloneqq C_0^G(M)\otimes\mathbb{C}^l\otimes\mathbb{C}^2,$$ 
and consider the cycle $[\Phi]\coloneqq [L,\Phi]\in KK(\mathbb{C},C_0^G(M))$. We have:
\begin{proposition}
	\label{evenKK}
	Let $E,E_0,D,D_0,B$ be as before, with $\dim M$ even. Then $$[\mathcal{E}^0,F]=[\Phi]\otimes_{C_0^G(M)}[D_{0}]\in KK(\mathbb{C},C^*(G)).$$
\end{proposition}
Next, we consider odd-dimensional $M$, with the initial Dirac operator $D_{0}$ being ungraded. Let $C^{1,0}$ be the Clifford algebra generated by a single element $X$ satisfying the relation $X^2 = -1$. Let $z$ be the grading of $C^{1,0}$, and form the operator 
$$D=D_{0}\otimes zX,$$ 
which defines a class $[D]\in KK(C_0^G(M)\otimes C^{1,0},C^*(G))$. Consider 
$$L\coloneqq C_0^G(M)\otimes\mathbb{C}^l\otimes C^{1,0}$$
as a $\mathbb{Z}_2$-graded Hilbert $C_0^G(M)\otimes C^{1,0}$-module. Write
$$\Phi=i(1-2P)\otimes X,$$ 
which defines a bounded adjointable operator on $L$. One verifies as in \cite{Bunke} 2.3.3 that the pair $(L,\Phi)$ defines a Kasparov module 
$$[\Phi]\coloneqq [L,\Phi]\in KK(\mathbb{C},C_0^G(M)\otimes C^{1,0}).$$

\noindent Let $\tau_{C^{1,0}}$ be the isomorphism
$$KK(\mathbb{C},C^*(G))\xrightarrow{\sim}KK(\mathbb{C}\otimes C^{1,0},C^*(G)\otimes C^{1,0})$$
defined in \cite{Blackadar} 17.8.5. Then we have:
\begin{proposition}
	\label{oddKK}
	Let $E,D,B$ be as before, with $\dim M$ odd. Then $$[\mathcal{E}^0,F]=\tau_{C^{1,0}}^{-1}([\Phi]\otimes_{C_0(M)}[D])\in KK(\mathbb{C},C^*(G)).$$
\end{proposition}

\subsection{The Higson $G$-corona}
\label{subsec:Gcorona}
We now turn to the $K$-theory of the Higson $G$-corona $\overline{M}^G$. We are motivated by the work of Keesling \cite{Keesling} on $K$-theory of the Higson compactification of metric spaces. Recall that for $X$ a metric space and $\phi$ a function on $X$, the map $V_r(\phi):M\rightarrow\mathbb{R}^+$ defined by
$$V_r(\phi)(x)=\sup\{|\phi(y)-\phi(x)|:y\in B_r(x)\}$$
is called the variation of $\phi$ at scale $r$. Define $C_h(X)$ to be the $C^*$-algebra of all bounded continuous functions $\phi$ on $X$ such that, for every $r$, $V_r(\phi)\rightarrow 0$ at infinity. The maximal ideal space $\overline{X}^d$ of $C_h(X)$ is the Higson compactification of $X$. In \cite{Keesling}, it is proved that:
\begin{theorem}
	Let $(X,d)$ be a non-compact connected metric space with proper metric $d$. Then $\check{H}^1(\overline{X}^d)$ contains a subgroup isomorphic to $(\mathbb{R},+)$, where $\check{H}^1$ denotes the first \v{C}ech cohomology group.
\end{theorem}
\begin{corollary}
	Let $(X,d)$ be as in the above proposition. Then $K^1(\overline{X}^d)$ is uncountable.
\end{corollary}
\begin{proof}
	The Chern character gives an isomorphism
	$$\ch\colon K^1(\overline{X}^d)\otimes\mathbb{Q}\xrightarrow{\sim}\check{H}^{\textnormal{odd}}(\overline{X}^d,\mathbb{Q})\cong\check{H}^{\textnormal{odd}}(\overline{X}^d)\otimes\mathbb{Q}.$$
	By the previous theorem, $\check{H}^1(X)^d\otimes\mathbb{Q}$, and hence $K^1(\overline{X}^d)\otimes\mathbb{Q}$, is uncountable.
\end{proof}
Let us return to the setting of the smooth manifold $M$ and the Higson $G$-compactification $\overline{M}^G$. We can characterise $\overline{M}^G$ using functions with values in a compact submanifold $Y\subseteq\mathbb{R}^N$ for some $N$. Recall that $f\in C_g^{G,\infty}(M)$ if $f$ is bounded, smooth and $G$-invariant, and $\norm{df}_{T^*M}\in C_0^G(M)$.
\begin{definition}
	Let $M$ be a non-cocompact $G$-Riemannian manifold and $Y$ a compact submanifold of $\mathbb{R}^N$ for some $N\geq 0$. Let $\pi_i\colon\mathbb{R}^N\rightarrow\mathbb{R}$ be the projection map onto the $i$-th coordinate. Suppose $f\colon M\rightarrow Y$ is a $G$-invariant function. Then we say that $f\in C_g^{G}(M,Y)$ if $\pi_i\circ f\in C_g^G(M)$ for each $1\leq i\leq N$.
\end{definition}
The proofs of the next two propositions are adapted from the work of Keesling \cite{Keesling}.
\begin{proposition}
	Let $M$ be a complete non-cocompact $G$-Riemannian manifold. Let $Y\subseteq\mathbb{R}^N$ be a compact submanifold for some $N\geq 0$ and $f\colon M/G\rightarrow Y$ continuous. Then $f$ has a continuous extension to the Higson $G$-compactification $\overline{M}^G$ if and only if $f\in C_g^G(M,Y)$. Further, $\overline{M}^G$ is the unique such compactification of $M/G$.
\end{proposition}
\begin{proof}
	Let $f\in C_g^{G}(M,Y)$. Without loss of generality, since $Y$ is compact, we may take $Y$ to be a submanifold of $[0,1]^{N}$. For each $j$, let $\pi_j\colon[0,1]^{N}\rightarrow [0,1]$ be the projection onto the $j$-th coordinate. Then $\pi_j\circ f\in C_g^{G}(M)$ and hence can be extended to a continuous function $\overline{\pi_j\circ f}$ on $\overline{M}^G$. The function $\overline{f}\colon\overline{M}^G\rightarrow [0,1]^{N},$
	$x\mapsto\left(\overline{\pi_1\circ f}(x),\ldots,\overline{\pi_{N}\circ f}(x)\right)$ is a continuous extension of $f$ to $\overline{M}^G$, taking values in $Y\subseteq [0,1]^{N}$. On the other hand, suppose $f\colon M\rightarrow Y$ is a continuous map not in $C_g^{G}(M,Y)$. Then there exists some $j$ for which $\pi_j\circ f\notin C_g^{G}(M)$. Thus $\pi_j\circ f$ does not extend continuously to $\overline{M}^G$. It follows that $f$ does not extend continuously to $\overline{M}^G$, for otherwise $\pi_j\circ f$ would also.
	
	For uniqueness, suppose $\mathcal{C}$ is a compactification of $M/G$ with the above properties. Then the set of continuous functions $M/G\rightarrow \mathbb{C}$ that extend to $\mathcal{C}$ is the closed subring $C_g^G(M)\subseteq C_b^G(M).$ Taking maximal ideals of $C_g^G(M)$ with the convex-hull topology recovers $\overline{M}^G$.
\end{proof}
We apply this proposition to $Y=S^1\subseteq\mathbb{R}^2$. The first \v{C}ech cohomology $\check{H}^1(\overline{M}^G)$ can be identified with the group $[\overline{M}^G,S^1]$ of homotopy classes of maps $\overline{M}^G\rightarrow S^1$, with a certain group operation derived from pointwise multiplication of functions. Let $e\colon\mathbb{R}\rightarrow S^1$ be the covering map $x\mapsto e^{2\pi i x}\in S^1\subseteq\mathbb{C}$. Let the algebra $C_g^{G,u}(M)$ be defined by the same conditions as $C_g^G(M)$ except that the functions need not be bounded. 

For clarity, let us denote $C_g^{G,b}(M)\coloneqq C_g^G(M)$, and let $C_g^{G,b}(M,\mathbb{R})$ and $C_g^{G,u}(M,\mathbb{R})$ be the real-valued functions in $C_g^{G,b}(M)$ and $C_g^{G,u}(M)$.

If $f\in C_g^{G,u}(M,\mathbb{R})$, it is easy to see that $e\circ f\in C_g^G(M,S^1)$, and so by the above proposition, $f$ has an extension $\overline{e\circ f}\colon\overline{M}^G\rightarrow S^1$. Define the map 	$a\colon C_g^{G,u}(M)\rightarrow\check{H}^1(\overline{M}^G)$ by
$$f\mapsto [\overline{e\circ f}]\in [\overline{M}^G,S^1]\cong\check{H}^1(\overline{M}^G).$$
\begin{proposition}
	There is an exact sequence of abelian groups 
	$$0\rightarrow C_g^{G,b}(M,\mathbb{R})\hookrightarrow C_g^{G,u}(M,\mathbb{R})\xrightarrow{a}\check{H}^1(\overline{M}^G).$$
\end{proposition}
\begin{proof}
	Clearly the sequence is exact at $C_g^{G,b}(M,\mathbb{R})$. To show that $C_g^{G,b}(M,\mathbb{R})\subseteq\ker(a)$, let $f\in C_g^{G,b}(M,\mathbb{R})$. Then $f$ has an extension $\overline{f}\colon\overline{M}^G\rightarrow\mathbb{R}$. Since $\overline{f}$ is a lift for $\overline{e\circ f}$, the latter map is null-homotopic. Next, to show that $\ker(a)\subseteq C_g^{G,b}(M,\mathbb{R})$, suppose $f\in C_g^{G,u}(M,\mathbb{R})$ and $a(f)=0\in [\overline{M}^G,S^1]$. This implies that $a(f) = [\overline{e\circ g}]$ is null-homotopic and hence $\overline{e\circ f}$ has a lift $q\colon\overline{M}^G\rightarrow\mathbb{R}$. Since both $q|_{M/G}$ and $f$ are lifts for $\overline{e\circ f}|_{M/G}$, they must be related by a deck transformation. Since $q$ is bounded, $f$ must be bounded, and $f\in C_g^{G,b}(M,\mathbb{R})$. Hence $\ker(a)=C_g^{G,b}(M,\mathbb{R})$.
\end{proof}
\begin{corollary}
	$\check{H}^1(\overline{M}^G)$ contains a subgroup isomorphic to $C_g^{G,u}(M,\mathbb{R})/C_g^{G,b}(M,\mathbb{R})$.
\end{corollary}
We now show that $C_g^{G,u}(M,\mathbb{R})/C_g^{G,b}(M,\mathbb{R})$ is uncountable. Since the action of $G$ on $M$ is proper and isometric, $M/G$ has a natural distance function $d_{M/G}$ given by the minimal geodesic distance between orbits (see \cite{Alexandrino} p.~74). $d_{M/G}$ lifts to a function $d_M$ on $M$ that is well-defined on pairs of orbits. For a fixed point $x_0\in M$, the function
\begin{align*}
d_{x_0}\colon M &\rightarrow\mathbb{R},\qquad x\mapsto d_M(x_0,x)
\end{align*}
is $G$-invariant and $1$-Lipschitz. We now show that $d_{x_0}$ can be approximated by a smooth $G$-invariant Lipschitz function.
\begin{proposition}
	For any $\epsilon > 0$, there is a smooth $G$-invariant function $d_{x_0,\epsilon}$ on $M$ such that for any $x\in M$,
	$$|d_{x_0,\epsilon}(x) - d_{x_0}(x)| < \epsilon,\qquad\norm{d(d_{x_0,\epsilon})}_{T^*M} \leq 1 + \epsilon.$$
\end{proposition}
\begin{proof}
	Since $M$ is a complete Riemannian manifold, it is a proper metric space and hence separable. By \cite{Azagra}, for any function $\delta:M\rightarrow(0,\infty)$ and $r>0$, one can construct a smooth approximation $f_{x_0}$ such that for all $x\in M$,
	$$|f_{x_0}(x)-d_{x_0}(x)|<\delta(x),\qquad\norm{df_{x_0}(x)}_{T^*M}<1+r.$$ 
	Since the $G$-action is proper, we can find a cut-off function $\mathfrak{c}$ on $M$. Let $\tilde{f}_{x_0}$ be the $G$-average of $f_{x_0}$, defined by $\tilde{f}_{x_0}(x)\coloneqq \int_G \mathfrak{c}(g^{-1}x)f_{x_0}(g^{-1}x)\,dg.$ Clearly $\tilde{f}_{x_0}$ is $G$-invariant. We now argue that $\tilde{f}_{x_0}$ is Lipschitz. Observe that $\norm{d\tilde{f}_{x_0}}_{T^*M}$ is equal to
	\begin{align*}
	&\qquad\norm{\int_G\left(d\mathfrak{c}(g^{-1}x)\right)f_{x_0}(g^{-1}(x))+\mathfrak{c}(g^{-1}x)df_{x_0}(g^{-1}x)\,dg}_{T^*M}\qquad\\
	\leq&\norm{\int_G d\mathfrak{c}(g^{-1}x)\left(d(g^{-1}(x))+\delta_x\right)dg}_{T^*M}+\int_G\mathfrak{c}(g^{-1}x)\norm{df_{x_0}(g^{-1}x)}_{T^*M}\,dg,
	\end{align*}
	where $\delta_x\coloneqq f_{x_0}(g^{-1}(x))-d(g^{-1}(x))$. Suppose we choose $\delta_x\leq\delta$ uniformly for some constant $\delta>0$. Then the above expression is bounded by
	$$\norm{\int_G\left(d\mathfrak{c}(g^{-1}x)\right)d(g^{-1}(x))\,dg}_{T^*M} + \delta\int_G\norm{d\mathfrak{c}(g^{-1}x)}_{T^*M}\,dg$$
	$$+\int_G\mathfrak{c}(g^{-1}x)\norm{df_{x_0}(g^{-1}x)}_{T^*M}\,dg.$$
	The first summand vanishes because $d$ is $G$-invariant, since for any $G$-invariant function $l\colon M\rightarrow\mathbb{R}$, we have 
	$$\int_G d\mathfrak{c}(g^{-1}x)l(g^{-1}x)\,dg=l(x)\int_G d\mathfrak{c}(g^{-1}x)\,dg=l(x)d(1)=0.$$ 
	For a fixed $x\in M$, let $\textnormal{supp}\,\mathfrak{c}(g^{-1}x)$ be the support in $G$ of the function $g\mapsto\mathfrak{c}(g^{-1}x)$. Let $U$ be a $G$-stable, cocompact subset of $M$, and let $\delta_U$ be an upper-bound for $\delta$ on $U$. For $x\in U$,
	\begin{align*}
	\delta_U\int_G\norm{d\mathfrak{c}(g^{-1}x)}_{T^*M}\,dg&\leq\delta_U\left(\sup_{x\in U}\norm{d\mathfrak{c}(x)}\right) \left(\sup_{x\in U}\,\left\{\textnormal{vol}\left(\textnormal{supp}\,\mathfrak{c}(g^{-1}x)\right)\right\}\right).
	\end{align*}
	The function $x\mapsto d\mathfrak{c}(x)$ has cocompactly compact support and hence is bounded above on $U$. The function $x\mapsto\textnormal{vol}\,(\textnormal{supp}\,\mathfrak{c}(g^{-1}x))$ is $G$-invariant and so descends to a compact set in $M/G$, whence it is also bounded above on $U$. The above product is bounded by $\delta_UC_U$, where the constant $C_U>0$ depends only on the cocompact set $U$. The third summand above is bounded by $1+r$, given our choice of $f_{x_0}$. Thus the whole expression is strictly less than $1+\delta_U C_U+r$. By picking $r<\frac{\epsilon}{2}$ and $\delta_U$ such that $\delta_U C_U<\frac{\epsilon}{2}$, this expression can be made to be less than $1+\epsilon$. By further choosing $\delta_U\leq\epsilon$, we have that for $x\in U$,
	\begin{align*}
	\left|\tilde{f}_{x_0}(x) - d_{x_0}(x)\right| &= \int_G \mathfrak{c}(g^{-1}x)\left|f_{x_0}(g^{-1}x)-d_{x_0}(g^{-1}x)\right|dg\leq\delta_U\leq\epsilon.
	\end{align*}
	To obtain the estimate on all of $M$, let $\mathcal{U}=\{U_i\colon i\in\mathbb{N}\}$ be a locally finite, countable open cover of $M$ by $G$-stable cocompact sets. There exists a smooth partition of unity subordinate to $\mathcal{U}$ consisting of $G$-invariant functions $\psi_{U_i}$ \cite{Palais}. Then 
	$$\delta(x)\coloneqq \sum_{i=1}^\infty\psi_i(x)\delta_{U_i}(x)$$ 
	is a well-defined smooth function $M\rightarrow\mathbb{R}$, so we may choose the approximation $f_{x_0}$ so that for all $x\in M$, $|f_{x_0}-d_{x_0}|<\delta(x)$. For each $x\in M$, let 
	$$C_x\coloneqq \max\{C_i\colon x\in U_i,\,i\in\mathbb{N}\}.$$ 
	Then it holds that for all $x\in M$, $\delta(x)C_x<\frac{\epsilon}{2}.$ By our previous calculations, for all $x\in M$ we have 
	$$\left|\tilde{f}_{x_0}(x) - d_{x_0}(x)\right|\leq\epsilon,$$ 
	and $\norm{d\tilde{f}_{x_0}}_{T^*M}\leq 1+\epsilon.$ Finally, set $d_{x_0,\epsilon} \coloneqq \tilde{f}_{x_0},$ and we conclude.
\end{proof}

\begin{remark} This argument applies more generally to produce a $G$-invariant smooth approximation $\tilde{f}$ starting from a $G$-invariant Lipschitz function $f$.
\end{remark}

\begin{proposition}
	Let $M$ be a non-cocompact $G$-Riemannian manifold. Then $\check{H}^1(\overline{M}^G)$ contains a subgroup isomorphic to $(\mathbb{R},+)$.
\end{proposition}
\begin{proof}
	Fix $x_0\in M$. Define $d_M, d_{M/G}, d_{x_0}$ and $d_{x_0,\epsilon}$ as in the proof of the previous proposition, so that for $y\in M$, 
	$$|d_{x_0,\epsilon}(y) - d_{x_0}(y)| < \epsilon, \qquad \norm{d(d_{x_0,\epsilon})(y)}_{T^*M} < 2.$$ 
	For each $r\in\mathbb{R}$, consider the function $\rho_r\in C_g^{G,u}(M)$ defined by $$\rho_r\colon M\rightarrow\mathbb{R},\qquad y\mapsto r\ln{d_{x_0,\epsilon}(y)}.$$ 
	If $r\neq s$, $\rho_r-\rho_s$ is unbounded. Hence the subgroup 
	$$\{[\rho_r]\colon r\in\mathbb{R}\}\subseteq C_g^{G,u}(M,\mathbb{R})/C_g^{G,b}(M,\mathbb{R})$$ 
	is uncountable; thus $C_g^{G,u}(M,\mathbb{R})/C_g^{G,b}(M,\mathbb{R})$ has rank at least $2^{\aleph_0}$. Now $(M,d_M)$ is a proper non-compact metric space, hence separable. Thus $(M/G, d_{M/G})$ is separable. Thus the rank of $C_g^{G,u}(M,\mathbb{R})$, and hence that of $$C_g^{G,u}(M,\mathbb{R})/C_g^{G,b}(M,\mathbb{R}),$$ is at most $2^{\aleph_0}$. Since $C_g^{G,u}(M,\mathbb{R})/C_g^{G,b}(M,\mathbb{R})$ is an abelian, divisible and torsion-free group, it is isomorphic to $(\mathbb{R},+)$.
\end{proof}
\begin{corollary}
	Let $M$ be a complete non-cocompact $G$-Riemannian manifold. Then both $K_1(C_g^G(M))$ and $K_1(C(\partial_h^{G}(\overline{M})))$ are uncountable.
\end{corollary}

\begin{proof}
	The first statement follows from the Chern character isomorphism
	$$\ch\colon K^1(\overline{M}^G)\otimes\mathbb{Q}\xrightarrow{\sim}\check{H}^1(\overline{M}^G,\mathbb{Q})\cong \check{H}^1(\overline{M}^G)\otimes\mathbb{Q}$$
	together with the above proposition, noting that $K_1(C_g^G(M))\cong K^1(\overline{M}^G)$. The second statement follows from the first by the Five Lemma.
\end{proof}
This shows that when $M$ is a complete non-cocompact $G$-Riemannian manifold, the group $K_1(C(\partial_h^{G}(\overline{M})))$ is uncountable. On the other hand $K_0(C_g^G(M)))\cong K^0(\overline{M}^G)$ contains a copy of $\mathbb{Z}$, since $\overline{M}^G$ is a compact space, so that $K^0(\partial_h^G(\overline{M}))$ is also infinite. In summary:
\begin{theorem}
	\label{thm:infiniteKofCorona}
	Let $M$ be a complete non-cocompact $G$-Riemannian manifold. Then the $K$-theory of the Higson $G$-corona of $M$ is uncountable.
\end{theorem}
	\vspace{1cm}

\section{Invariant Metrics of Positive Scalar Curvature}
\label{sec:PSC}
Let $M$ be a $G$-spin manifold with spinor bundle $S$ and Dirac operator $\dirac_0$. Let $\Phi$ be a $G$-admissible endomorphism, and form the $G$-Callias-type operator $B=\dirac+\Phi$, where $\dirac$ is a $\mathbb{Z}_2$-graded version of $\dirac_0$ acting on the $\mathbb{Z}_2$-graded bundle $E$ constructed from $S=E_0$, in the notation of section \ref{sec:G-Callias}. Form the $\mathbb{Z}_2$-graded $G$-Sobolev modules $\mathcal{E}^i=\mathcal{E}^i(E)$ as in section \ref{sec:G-Sobolev}. Then by Proposition \ref{prop:Bbounded}, $B$ extends to a bounded adjointable operator $\mathcal{E}^1\rightarrow\mathcal{E}^0$. For $\lambda\in\mathbb{R}$, let 
$$R(\lambda) = (B^2+f+\lambda^2)^{-1},$$ 
which exists by Lemma \ref{resolventlemma}. Normalising $B$ gives rise to the operator $F \coloneqq  BR(0)^{1/2}\in\mathcal{L}(\mathcal{E}^0)$, by subsection \ref{subsec:equivalentdef}.

Denote the $G$-invariant Riemannian metric on $M$ by $g$, not to be confused with elements of the group $G$. Suppose the scalar curvature $\kappa^g$ associated to $g$ is everywhere positive. By Lichnerowicz's formula, 
$$B^2 = \dirac^2+\dirac\Phi+\Phi\dirac+\Phi^2 = \nabla^{*}\nabla+\frac{\kappa^g}{4}+\dirac\Phi+\Phi\dirac+\Phi^2,$$
where $\nabla$ is the connection on $E$. We now show that $g$ can be scaled by an appropriate constant to make $B$ invertible. For $r>0$, define the metric $rg$ on $M$ by $rg(v,w)\coloneqq r\cdot g(v,w)$, and let $D^{rg}$ be the associated Dirac operator. The $G$-admissible endomophism $\Phi$ is still $G$-admissible for $D^{rg}$, so the operator $B^{rg}\coloneqq D^{rg}+\Phi$ is again of $G$-Callias-type.

\begin{proposition}
	\label{prop:scaling}
	There exists $r>0$ for which $B^{rg}$ is strictly positive.
\end{proposition}
\begin{proof}
	$\dirac\Phi + \Phi\dirac$ is bounded away from $0$ outside a cocompact subset $K\subseteq M$. It is equal to Clifford multiplication by a one-form $dR$, where $R$ is a projection or unitary depending on $\dim M$, and is bounded below by a constant $C_\Phi$ depending on $\Phi$. Now if $\frac{\kappa^g}{4}>-C_\Phi+\epsilon$ 
	everywhere on $K$ for some $\epsilon > 0$, then $B^2$ is strictly positive. Otherwise, note that $\kappa^{rg} = \frac{\kappa^g}{r^2}$, and $c^{rg} = \frac{c}{r}$, where $c$ means Clifford multiplication by a one-form. The latter implies that $C^{rg}_\Phi = \frac{C_\Phi}{r},$ since, for a fixed $\Phi$, the constant $C_\Phi$ scales with the metric in the same way as $c$. 
	As $K$ is cocompact, $\kappa^g$ is bounded below by some $\kappa_0>0$ on $K$. Thus we can find $r>0$ such that $\frac{\kappa^{rg}}{4}>-C^{rg}_\Phi+\epsilon.$ It follows that $(B^{rg})^2$ is a strictly positive operator.
\end{proof}
\begin{proposition}
	\label{prop:invertibleCallias}
	Let $B$ be a $G$-Callias-type operator with $B^2$ strictly positive. Then $B$ has a bounded adjointable inverse.
\end{proposition}
\begin{proof}
	It suffices to show that $B^2\colon\mathcal{E}^2\rightarrow\mathcal{E}^0$ is invertible. By regularity of $B$ and hence $B^2$, $B^2+\mu^2$ is surjective for every positive number $\mu^2$. Further, since $B^2$ is strictly positive, $B^2 + \mu^2$ is injective with bounded inverse. To see that $(B^2+\mu^2)^{-1}$ is adjointable, let $T\coloneqq (B^2+\mu^2)^{-1}$. Then $T$ is self-adjoint as a bounded operator $\mathcal{E}^0\rightarrow\mathcal{E}^0$, which follows from Lemma 4.1 in \cite{Lance} and $\langle u,Tu\rangle_{\mathcal{E}^0} = \langle(B+\mu^2)Tu,Tu\rangle_{\mathcal{E}^0}\geq\mu^2\langle Tu,Tu\rangle_{\mathcal{E}^0}\geq 0.$ Next, for any $w\in\mathcal{E}^0$ and $u\in\mathcal{E}^2$, we have
	\begin{align*}
	\langle T u,w\rangle_{\mathcal{E}^2} &= \langle B^2 Tu,B^2 w\rangle_{\mathcal{E}^0} +\langle B Tu,B w\rangle_{\mathcal{E}^0} + \langle T u, w\rangle_{\mathcal{E}^0}\\
	&=\langle(B^2+\mu^2)Tu,B^2 w\rangle_{\mathcal{E}^0} + (1-\mu^2)\langle Tu,B^2 w\rangle_{\mathcal{E}^0}+ \langle u,Tw\rangle_{\mathcal{E}^0}\\
	&=\langle u,B^2 w\rangle_{\mathcal{E}^0} + (1-\mu^2)\langle u,TB^2 w\rangle_{\mathcal{E}^0} + \langle u,T w\rangle_{\mathcal{E}^0}\\
	&= \langle u,(B^2 + (1- \mu^2)TB^2 + B)w\rangle_{\mathcal{E}^0},
	\end{align*}
	Thus $(B^2+\mu^2)^{-1}\in\mathcal{L}(\mathcal{E}^0,\mathcal{E}^2)$. Now note that $B^2 = (1-\mu^2(B^2+\mu^2)^{-1})(B^2+\mu^2).$ 
	Since $B^2$ is a strictly positive operator, there exists $C>0$ such that for all $s\in\mathcal{E}^2$, we have $\langle B^2 s,s\rangle_{\mathcal{E}^0}\geq C\langle s,s\rangle_{\mathcal{E}^0}$. It follows from the Cauchy-Schwarz inequality for Hilbert modules that for any $\psi\in\mathcal{E}^0$, $$\norm{\mu^2(B^2+\mu^2)^{-1}\psi}_{\mathcal{E}^0}\leq\frac{\mu^2}{\mu^2+C}\norm{(B^2+\mu^2)T\psi}_{\mathcal{E}^0} = \frac{\mu^2}{\mu^2+C}\norm{\psi}_{\mathcal{E}^0}.$$ 
	Hence $(1-\mu^2(B^2+\mu^2)^{-1})$ has an adjointable inverse, and
	\begin{align*}
	(B^2)^{-1} &= T(1-\mu^2(B+\mu^2)^{-1})^{-1}.\qedhere
	\end{align*}
\end{proof}
This yields the following application to $G$-invariant metrics of positive scalar curvature:
\begin{theorem}
	\label{thm:obstruction}
	Let $M$ be a $G$-spin manifold with Dirac operator $\dirac_0$. Suppose $M$ admits a $G$-invariant positive scalar curvature metric. Let $D$ be the $\mathbb{Z}_2$-graded Dirac operator formed from $D_0=\dirac_0$ as in section \ref{sec:G-Callias}. Then $\ind_G F=0\in K_0(C^*(G)),$ where $F$ is the bounded transform of any $G$-Callias-type operator defined by a $G$-admissible $\Phi$.
\end{theorem}
\begin{proof}
	Since $B$ is $G$-invertible at infinity, it follows from the results in section \ref{sec:G-Invertible} that $(\mathcal{E}^0,F)$ defines a class
	$$[B]\coloneqq [\mathcal{E}^0,F]\in KK(\mathbb{C},C^*(G)).$$ 
	This class is independent of the choice of metric on $M$ (see \cite{RosenbergTrivial}). In particular, let $F$ and $F^{rg}$ be the normalised Callias-type operators associated to the metrics $g$ and $rg$ respectively, for some $r>0$. Since  $[\mathcal{E}^0,F]$ and $[\mathcal{E}^0,F^{rg}]$ are related by an element of $\mathcal{K}(\mathcal{E}^0)$, $\ind_G F=\ind_G F^{rg}$. By Propositions \ref{prop:scaling} and \ref{prop:invertibleCallias}, we can find an $r$ such that $F^{rg}=B^{rg}R^{rg}(0)^{1/2}$ is invertible.
\end{proof}
As immediate corollaries, we obtain:
\begin{corollary}[\cite{GMW} Theorem 54]
	Suppose $M/G$ is compact, with $\dirac$ being the $G$-Spin-Dirac operator on $M$. Suppose that $M$ admits a Riemannian metric of positive scalar curvature. Then $
	{\rm index}_G(\dirac) = 0 \in K_*(C^*(G)).$
\end{corollary}
\begin{corollary}[\cite{WeipingShort} Theorem 1.2]
	Suppose $M/G$ is compact, with $\dirac$ being the $G$-Spin-Dirac operator on $M$. Then ${\rm ind}_G(\dirac)=0,$ where ${\rm ind}_G$ denotes the Mathai-Zhang index.
\end{corollary}

\bibliographystyle{abbrv}
\bibliography{G-Callias.bib}

\begin{thebibliography}{10}

\bibitem{Alexandrino}
M.~M. Alexandrino and R.~G. Bettiol.
\newblock {\em Lie groups and geometric aspects of isometric actions}.
\newblock Springer, Cham, 2015.

\bibitem{Border}
C.~D. Aliprantis and K.~C. Border.
\newblock {\em Infinite dimensional analysis}.
\newblock Springer, Berlin, third edition, 2006.
\newblock A hitchhiker's guide.

\bibitem{Anghel}
N.~Anghel.
\newblock On the index of {C}allias-type operators.
\newblock {\em Geom. Funct. Anal.}, 3(5):431--438, 1993.

\bibitem{Azagra}
D.~Azagra, J.~Ferrera, F.~L\'opez-Mesas, and Y.~Rangel.
\newblock Smooth approximation of {L}ipschitz functions on {R}iemannian
  manifolds.
\newblock {\em J. Math. Anal. Appl.}, 326(2):1370--1378, 2007.

\bibitem{Blackadar}
B.~Blackadar.
\newblock {\em {$K$}-theory for operator algebras}, volume~5 of {\em
  Mathematical Sciences Research Institute Publications}.
\newblock Cambridge University Press, Cambridge, second edition, 1998.

\bibitem{BlackadarCStar}
B.~Blackadar.
\newblock {\em Operator algebras}, volume 122 of {\em Encyclopaedia of
  Mathematical Sciences}.
\newblock Springer-Verlag, Berlin, 2006.
\newblock Theory of $C^*$-algebras and von Neumann algebras, Operator Algebras
  and Non-commutative Geometry, III.

\bibitem{BottSeeley}
R.~Bott and R.~Seeley.
\newblock Some remarks on the paper of {C}allias: ``{A}xial anomalies and index
  theorems on open spaces''\ [{C}omm. {M}ath. {P}hys. {\bf 62} (1978), no. 3,
  213--234;\ {MR} 80h:58045a].
\newblock {\em Comm. Math. Phys.}, 62(3):235--245, 1978.

\bibitem{Braverman}
M.~Braverman.
\newblock The index theory on non-compact manifolds with proper group action.
\newblock {\em J. Geom. Phys.}, 98:275--284, 2015.

\bibitem{BruningMoscovici}
J.~Br\"uning and H.~Moscovici.
\newblock {$L^2$}-index for certain {D}irac-{S}chr\"odinger operators.
\newblock {\em Duke Math. J.}, 66(2):311--336, 1992.

\bibitem{Bunke}
U.~Bunke.
\newblock A {$K$}-theoretic relative index theorem and {C}allias-type {D}irac
  operators.
\newblock {\em Math. Ann.}, 303(2):241--279, 1995.

\bibitem{Callias}
C.~Callias.
\newblock Axial anomalies and index theorems on open spaces.
\newblock {\em Comm. Math. Phys.}, 62(3):213--234, 1978.

\bibitem{Cecchini}
S.~Cecchini.
\newblock {\em Callias-type {O}perators in {C}* -algebras}.
\newblock ProQuest LLC, Ann Arbor, MI, 2017.
\newblock Thesis (Ph.D.)--Northeastern University.

\bibitem{Ebert}
J.~Ebert.
\newblock Index theory in spaces of noncompact manifolds {I}: Analytical
  foundations.
\newblock {\em Preprint.}, 2016.

\bibitem{GMW}
H.~Guo, V.~Mathai, and H.~Wang.
\newblock Positive scalar curvature and poincar\'e duality for proper actions.
\newblock {\em Submitted, https://arxiv.org/abs/1609.01404}, 2016.

\bibitem{HochsQuantisation}
P.~Hochs.
\newblock Quantisation commutes with reduction at discrete series
  representations of semisimple groups.
\newblock {\em Adv. Math.}, 222(3):862--919, 2009.

\bibitem{HochsLandsman}
P.~Hochs and N.~P. Landsman.
\newblock The {G}uillemin-{S}ternberg conjecture for noncompact groups and
  spaces.
\newblock {\em J. K-Theory}, 1(3):473--533, 2008.

\bibitem{HochsMathai}
P.~Hochs and V.~Mathai.
\newblock Geometric quantization and families of inner products.
\newblock {\em Adv. Math.}, 282:362--426, 2015.

\bibitem{HochsMathai2}
P.~Hochs and V.~Mathai.
\newblock Quantising proper actions on {S}pin{$^c$}-manifolds.
\newblock {\em Asian J. Math.}, 21(4):631--685, 2017.

\bibitem{Hochs-Song2}
P.~Hochs and Y.~Song.
\newblock An equivariant index for proper actions {II}.
\newblock {\em J. Noncommut. Geom.}, To appear.

\bibitem{Hochs-Song3}
P.~Hochs and Y.~Song.
\newblock An equivariant index for proper actions {III}: {T}he invariant and
  discrete series indices.
\newblock {\em Differential Geom. Appl.}, 49:1--22, 2016.

\bibitem{Hochs-Song1}
P.~Hochs and Y.~Song.
\newblock An equivariant index for proper actions {I}.
\newblock {\em J. Funct. Anal.}, 272(2):661--704, 2017.

\bibitem{KaadLesch}
J.~Kaad and M.~Lesch.
\newblock A local global principle for regular operators in {H}ilbert
  {$C^*$}-modules.
\newblock {\em J. Funct. Anal.}, 262(10):4540--4569, 2012.

\bibitem{Kankaanrinta}
M.~Kankaanrinta.
\newblock Equivariant collaring, tubular neighbourhood and gluing theorems for
  proper {L}ie group actions.
\newblock {\em Algebr. Geom. Topol.}, 7:1--27, 2007.

\bibitem{Kasparov}
G.~Kasparov.
\newblock Elliptic and transversally elliptic index theory from the viewpoint
  of {$KK$}-theory.
\newblock {\em J. Noncommut. Geom.}, 10(4):1303--1378, 2016.

\bibitem{Keesling}
J.~Keesling.
\newblock The one-dimensional \v cech cohomology of the {H}igson
  compactification and its corona.
\newblock {\em Topology Proc.}, 19:129--148, 1994.

\bibitem{Kucerovsky}
D.~Kucerovsky.
\newblock A short proof of an index theorem.
\newblock {\em Proc. Amer. Math. Soc.}, 129(12):3729--3736, 2001.

\bibitem{Kustermans}
J.~Kustermans.
\newblock The functional calculus of regular operators on hilbert $c^*$-modules
  revisited.
\newblock {\em Preprint Odense Universitet.}, 1997.

\bibitem{Lance}
E.~C. Lance.
\newblock {\em Hilbert {$C^*$}-modules}, volume 210 of {\em London Mathematical
  Society Lecture Note Series}.
\newblock Cambridge University Press, Cambridge, 1995.
\newblock A toolkit for operator algebraists.

\bibitem{Landsman}
N.~P. Landsman.
\newblock Functorial quantization and the {G}uillemin-{S}ternberg conjecture.
\newblock In {\em Twenty years of {B}ialowieza: a mathematical anthology},
  volume~8 of {\em World Sci. Monogr. Ser. Math.}, pages 23--45. World Sci.
  Publ., Hackensack, NJ, 2005.

\bibitem{Manuilov}
V.~M. Manuilov and E.~V. Troitsky.
\newblock {\em Hilbert {$C^*$}-modules}, volume 226 of {\em Translations of
  Mathematical Monographs}.
\newblock American Mathematical Society, Providence, RI, 2005.
\newblock Translated from the 2001 Russian original by the authors.

\bibitem{MZ}
V.~Mathai and W.~Zhang.
\newblock Geometric quantization for proper actions.
\newblock {\em Adv. Math.}, 225(3):1224--1247, 2010.
\newblock With an appendix by Ulrich Bunke.

\bibitem{Neeb}
K.-H. Neeb.
\newblock Current groups for non-compact manifolds and their central
  extensions.
\newblock In {\em Infinite dimensional groups and manifolds}, volume~5 of {\em
  IRMA Lect. Math. Theor. Phys.}, pages 109--183. de Gruyter, Berlin, 2004.

\bibitem{Palais}
R.~S. Palais.
\newblock On the existence of slices for actions of non-compact {L}ie groups.
\newblock {\em Ann. of Math. (2)}, 73:295--323, 1961.

\bibitem{PalaisBook}
R.~S. Palais and C.-L. Terng.
\newblock {\em Critical point theory and submanifold geometry}, volume 1353 of
  {\em Lecture Notes in Mathematics}.
\newblock Springer-Verlag, Berlin, 1988.

\bibitem{Pettis}
B.~J. Pettis.
\newblock On integration in vector spaces.
\newblock {\em Trans. Amer. Math. Soc.}, 44(2):277--304, 1938.

\bibitem{Roe}
J.~Roe.
\newblock Positive curvature, partial vanishing theorems and coarse indices.
\newblock {\em Proc. Edinb. Math. Soc. (2)}, 59(1):223--233, 2016.

\bibitem{RosenbergTrivial}
J.~Rosenberg.
\newblock The {$K$}-homology class of the {E}uler characteristic operator is
  trivial.
\newblock {\em Proc. Amer. Math. Soc.}, 127(12):3467--3474, 1999.

\bibitem{WeipingShort}
W.~Zhang.
\newblock A note on the {L}ichnerowicz vanishing theorem for proper actions.
\newblock {\em J. Noncommut. Geom.}, 11(3):823--826, 2017.

\end{thebibliography}

\end{document}